\documentclass[10pt]{amsart}
\usepackage{amsmath}
\usepackage{amssymb}
\usepackage{cases}
\usepackage{amscd}
\usepackage[colorlinks,
            linkcolor=black,
            anchorcolor=black,
            citecolor=black,
            urlcolor=black
            ]{hyperref}

 \newtheorem{thm}{Theorem}[section]

 \newtheorem{abcdef}{Theorem}[section]
 \newtheorem{lem}[thm]{Lemma}
 
 \newtheorem{abcd}{Theorem}[section]
 \theoremstyle{remark}
 \newtheorem{remark}[abcd]{Remark}
 \theoremstyle{definition}
 \newtheorem{definition}[abcdef]{Definition}
 \numberwithin{equation}{section}
\begin{document}

\title[Extension of cohomology classes]
{Extension of cohomology classes and holomorphic sections defined on subvarieties}

\author{Xiangyu Zhou, Langfeng Zhu}

\address{Xiangyu Zhou: Institute of Mathematics, AMSS, and Hua Loo-Keng Key Laboratory of Mathematics,
Chinese Academy of Sciences, Beijing 100190, China; School of
Mathematical Sciences, University of Chinese Academy of Sciences,
Beijing 100049, China}

\email{xyzhou@math.ac.cn}

\address{Langfeng Zhu: School of Mathematics and Statistics, Wuhan University, Wuhan 430072, China}

\email{zhulangfeng@amss.ac.cn}

\thanks{Xiangyu Zhou was partially supported by the National Natural Science Foundation of China (No. 11688101
and No. 11431013). Langfeng Zhu was partially supported by the National Natural Science Foundation
of China (No. 11201347 and No. 11671306).}



%%% --------------------------------------------------

\begin{abstract}
In this paper, we obtain two extension theorems for cohomology classes and holomorphic sections defined on analytic subvarieties, which are defined as the supports of
the quotient sheaves of multiplier ideal sheaves of quasi-plurisubharmonic functions with arbitrary singularities. The first result gives a positive answer to a question posed
by Cao-Demailly-Matsumura, and unifies a few well-known injectivity theorems. The second result generalizes and optimizes a general $L^2$ extension theorem obtained by Demailly.
\end{abstract}

%%% --------------------------------------------------
\maketitle
%%% --------------------------------------------------

%%% ------------------------------------------------

\section{Introduction and main results}\label{section-introduction}

Let $(X,\mathcal{O}_X)$ be a complex manifold, $\mathcal{J}\subset \mathcal{O}_X$ be a coherent ideal sheaf, and $L$ be a holomorphic line bundle on $X$. Let $Y:=V(\mathcal{J})$ be the zero variety of $\mathcal{J}$ endowed with the structure sheaf $\mathcal{O}_Y:=(\mathcal{O}_X/\mathcal{J})\big|_{Y}$. Then $(Y,\mathcal{O}_Y)$ may be non reduced.

The extension problem for cohomology classes defined on $(Y,\mathcal{O}_Y)$ is to find appropriate conditions on $X$, $\mathcal{J}$ and $L$ such that the natural homomorphism
\[H^q\big(X,\mathcal{O}_X(K_X\otimes L)\big)\longrightarrow H^q\big(Y,\mathcal{O}_Y(K_X\otimes L)\big)=H^q\big(X,\mathcal{O}_X(K_X\otimes L)\otimes\mathcal{O}_X/\mathcal{J}\big)\]
is surjective, where $q$ is a nonnegative integer. The surjectivity is equivalent to the injectivity of the homomorphism
\[H^{q+1}\big(X,\mathcal{O}_X(K_X\otimes L)\otimes\mathcal{J}\big)\longrightarrow H^{q+1}\big(X,\mathcal{O}_X(K_X\otimes L)\big).\]

This extension problem is related to vanishing theorems and injectivity theorems. If $X$ is Stein, this problem is solved by Cartan's Theorem B. For other cases, there have been a great number of important works related to this problem, and various advanced techniques have been developed.

In this paper, we present and develop some new idea and technique based on existed advanced techniques to obtain new results on this problem. There are two main points which seem to be different from the previous related works. One is that we introduce an idea to approximate two weight functions simultaneously (see Lemma \ref{l:regularization-two-quasipsh}, Subsection \ref{subsection 2} and Subsection \ref{jet-subsection 2}). Another one is that we take an idea such that the limit process for weight functions is done prior to other limit processes after solving $\bar\partial$-equations (see Subsection \ref{subsection 4} and Subsection \ref{jet-subsection 4}).\\

The first result in this paper (Theorem \ref{t:Zhou-Zhu-nonreduced-quasipsh}) is an extension theorem for cohomology classes defined on a not necessarily reduced analytic subvariety, which is defined as the support of the quotient sheaf of multiplier ideal sheaves
of quasi-plurisubharmonic functions with arbitrary singularities. This result gives a positive answer to a question posed by Cao-Demailly-Matsumura in \cite{Cao-Demailly-Matsumura2017} (see Remark \ref{r:C-D-M question} below), and also unifies a few well-known injectivity theorems (see Remark \ref{r:corollaries-thm-1} below).\\

Let us first recall some definitions (see \cite{Nadel1990}, \cite{Demailly92a}, \cite{Demailly2001}, \cite{Siu2005}, \cite{Demailly-book2010}, \cite{Ohsawa-book2015}, \cite{Demailly2015}, \cite{Cao-Demailly-Matsumura2017}, \cite{Demailly2017}, etc).

Let $X$ be a complex manifold. A function $\varphi:\,X\longrightarrow[-\infty,+\infty)$ on $X$ is said to be \textbf{quasi-plurisubharmonic} (quasi-psh) if $\varphi$ is locally the sum of a plurisubharmonic function and a smooth function.

If $\varphi$ is a quasi-psh function on $X$, the \textbf{multiplier
ideal sheaf} $\mathcal{I}(e^{-\varphi})$ is the ideal subsheaf of $\mathcal{O}_X$ defined by
\[\mathcal{I}(e^{-\varphi})_x=\{f\in \mathcal{O}_{X,x}:\,\exists\, U\ni x\text{ such that }\int_U|f|^2e^{-\varphi}d\lambda<+\infty\},\]
where $U$ is an open coordinate neighborhood of $x$, and $d\lambda$ is the Lebesgue measure with respect to the coordinates on $U$. It is well known that $\mathcal{I}(e^{-\varphi})$ is coherent.

Let $L$ be a holomorphic line bundle over $X$. A \textbf{singular Hermitian metric} $h$ on $L$ is simply a Hermitian metric which can be expressed locally as $e^{-\varphi_U}$ on $U$ such that $\varphi_U$ is quasi-psh, where $U\subset X$ is a local coordinate chart such that $L|_U\simeq U\times\mathbb{C}$.
It has a well-defined curvature current $\sqrt{-1}\Theta_{L,h}:=\sqrt{-1}\partial\bar\partial\varphi_U$ on $X$.

In the definition of a singular Hermitian metric, $\varphi_U$ is required to be quasi-psh here, while $\varphi_U$ is only required to be in $L^1_{\mathrm{loc}}$ in \cite{Demailly92a} and \cite{Demailly-book2010}.\\

Our first result is the following extension theorem, which was announced in \cite{Zhou-Zhu2018a}.

\begin{thm}\label{t:Zhou-Zhu-nonreduced-quasipsh}
Let $X$ be a holomorphically convex complex $n$-dimensional
manifold possessing a K\"{a}hler metric $\omega$, $\psi$ be an $L^1_{\mathrm{loc}}$ function on $X$ which is locally bounded above, and $(L,h)$ be a holomorphic line bundle over $X$ equipped with a
singular Hermitian metric $h$. Assume that $\alpha>0$ is a positive continuous function
 on $X$, and that the following two inequalities hold on $X$ in the sense of currents:
\\$(i)$\quad$\sqrt{-1}\Theta_{L,h}+\sqrt{-1}\partial\bar\partial\psi\geq0$,
\\$(ii)$\quad$\sqrt{-1}\Theta_{L,h}+(1+\alpha)\sqrt{-1}\partial\bar\partial\psi\geq0$.
\\Then the homomorphism induced by the natural inclusion $\mathcal{I}(he^{-\psi})\longrightarrow\mathcal{I}(h)$,
\[H^q\big(X,\mathcal{O}_X(K_X\otimes L)\otimes\mathcal{I}(he^{-\psi})\big)\longrightarrow H^q\big(X,\mathcal{O}_X(K_X\otimes L)\otimes \mathcal{I}(h)\big)\]
is injective for every $q\geq0$. In other words, the homomorphism induced by the natural sheaf surjection $\mathcal{I}(h)\longrightarrow\mathcal{I}(h)/\mathcal{I}(he^{-\psi})$,
\begin{equation*}
H^q\big(X,\mathcal{O}_X(K_X\otimes L)\otimes\mathcal{I}(h)\big)\longrightarrow H^q\big(X,\mathcal{O}_X(K_X\otimes L)\otimes\mathcal{I}(h)/\mathcal{I}(he^{-\psi})\big)
\end{equation*}
is surjective for every $q\geq0$.
\end{thm}

\begin{remark}\label{r:Y-definition}
The quotient sheaf $\mathcal{I}(h)/\mathcal{I}(he^{-\psi})$ is supported on an analytic subvariety $Y\subset X$, which is the zero set of the ideal sheaf
\[\mathcal{J}_Y:=\mathcal{I}(he^{-\psi}):\mathcal{I}(h)=\{g\in\mathcal{O}_X:\,\,g\cdot\mathcal{I}(h)\subset\mathcal{I}(he^{-\psi})\}.\]
The structure sheaf of $Y$ is $\mathcal{O}_Y:=(\mathcal{O}_X/\mathcal{J}_Y)\big|_{Y}$.

When $h$ is smooth, we have $\mathcal{I}(h)=\mathcal{O}_X$ and $\mathcal{I}(he^{-\psi})=\mathcal{I}(e^{-\psi})=\mathcal{J}_Y$. Then $\mathcal{O}_Y=\big(\mathcal{I}(h)/\mathcal{I}(he^{-\psi})\big)\big|_{Y}$ and Theorem \ref{t:Zhou-Zhu-nonreduced-quasipsh} implies that
\[H^q\big(X,\mathcal{O}_X(K_X\otimes L)\big)\longrightarrow H^q\big(Y,\mathcal{O}_Y(K_X\otimes L)\big)\]
is surjective.\qed
\end{remark}

\begin{remark}\label{r:C-D-M question}
Theorem \ref{t:Zhou-Zhu-nonreduced-quasipsh} was proved in \cite{Cao-Demailly-Matsumura2017} in the case when $\psi$ is a quasi-psh function with neat analytic singularities, here, a quasi-psh function $\varphi$ on $X$ is said to have analytic singularities if every point $x\in X$ possesses an open
neighborhood $U$ on which $\varphi$ can be written as
\[\varphi = c\log\sum\limits_{1\leq j\leq j_0}|g_j|^2+u,\]
where $c$ is a nonnegative number, $g_j\in\mathcal{O}_X(U)$ and $u$ is a bounded function on $U$. If $u$ is further assumed to be a smooth function on $U$, $\varphi$ is said to have neat analytic singularities (see \cite{Demailly2015}).

The general case when $\psi$ is a quasi-psh function with arbitrary singularities was posed as a question in Remark 3.10 in \cite{Cao-Demailly-Matsumura2017}. Theorem \ref{t:Zhou-Zhu-nonreduced-quasipsh} gives an affirmative answer to this question.\qed
\end{remark}

\begin{remark}\label{r:corollaries-thm-1}
Theorem \ref{t:Zhou-Zhu-nonreduced-quasipsh} also unifies some injectivity theorems in previous important works (see \cite{Tankeev1971}, \cite{Kollar1986}, \cite{Enoki1993}, \cite{Takegoshi1995}, \cite{Ohsawa2005}, \cite{Fujino2013a}, etc, especially the recent works \cite{Matsumura2014}, \cite{Matsumura2018}, \cite{Gongyo-Matsumura2017}, \cite{Fujino-Matsumura2016}, \cite{Matsumura2016}). In Section \ref{section-corollaries-of-thm-1}, we will show how to deduce these injectivity theorems from Theorem \ref{t:Zhou-Zhu-nonreduced-quasipsh}. A key point in deducing them is that $\psi$ is neither required to have analytic singularities nor required to be quasi-psh in Theorem \ref{t:Zhou-Zhu-nonreduced-quasipsh}. We will also discuss the application of Theorem \ref{t:Zhou-Zhu-nonreduced-quasipsh} to vanishing theorems in Section \ref{section-corollaries-of-thm-1}.\qed
\end{remark}

Some $L^2$ extension theorems and their important applications have been obtained since the establishment of the celebrated Ohsawa-Takegoshi $L^2$ extension theorem in \cite{Ohsawa-Takegoshi}. In recent years, optimal $L^2$ extension theorems (\cite{Blocki2012b}, \cite{Guan-Zhou2012}, \cite{Guan-Zhou2014-11}, \cite{Guan-Zhou2013b}) have been established since the utilization of the method of the undetermined function with ODE initiated in \cite{Guan-Zhou-Zhu-2011CR} and \cite{Zhu-Guan-Zhou} (see also \cite{Zhou-Zhu2010}). As an application of the optimal $L^2$ extension, inequality part of the Suita conjecture has been solved (\cite{Blocki2012b}, \cite{Guan-Zhou2012}). At that time very few other applications and connections of the optimal $L^2$ extension theorem existed, till some unexpected applications (including the proof of the full version of the Suita conjecture and the geometric meaning of the optimal $L^2$ extension) were found in \cite{Guan-Zhou2013b}. Actually, at the present time optimal $L^2$ extension theorems have many more interesting applications (see \cite{Cao2014}, \cite{Guan-Zhou2015-Invent}, \cite{Guan-Zhou2017}, \cite{Hacon-Popa-Schnell2016}, \cite{Ohsawa-book2015}, \cite{Ohsawa2017}, \cite{Paun-Takayama14}, \cite{Zhou-Zhu2015}, \cite{Zhou-Zhu2017}, etc).\\

For the above extension problem, it is desirable to obtain some optimal $L^2$ estimate. The second result in this paper (Theorem \ref{t:Zhou-Zhu-optimal-jet-extension}) is an $L^2$ extension theorem with an optimal estimate for holomorphic sections with an estimate, which generalizes and optimizes a general $L^2$ extension theorem in \cite{Demailly2015} (see Remark \ref{r:Demailly-jet} below). First let's recall some notions and notations as below.

Let $X$ be a complex $n$-dimensional manifold possessing a smooth Hermitian metric $\omega$, $\psi$ be an $L^1_{\mathrm{loc}}$ function on $X$ which is locally bounded above, and $(L,h)$ be a holomorphic line bundle over $X$ equipped with a singular Hermitian metric $h$. Assume that $\sqrt{-1}\Theta_{L,h}+\sqrt{-1}\partial\bar\partial\psi\geq\gamma$ on $X$ in the sense of currents for some continuous real $(1,1)$-form $\gamma$ on $X$.

In this paper, we don't assume that $\psi$ has analytic singularities and that $\psi$ is quasi-psh, although $\psi$ was assumed to be a quasi-psh function with analytic singularities in \cite{Ohsawa5}, \cite{Guan-Zhou2013b} and \cite{Demailly2015}.

Following Definition 2.11 in \cite{Demailly2015}, the \textbf{restricted multiplier ideal sheaf} $\mathcal{I}'_{\psi}(h)$ is defined
to be the set of germs $f\in\mathcal{I}(h)_x\subset\mathcal{O}_{X,x}$ such that there exists a coordinate neighborhood $U$ of $x$ satisfying
\[\varlimsup_{t\rightarrow-\infty}\int_{\{y\in U:\,t<\psi(y)<t+1\}}|f|^2e^{-\varphi-\psi}d\lambda<+\infty,\]
where $U$ is small enough such that $h$ can be written as $e^{-\varphi}$ with respect to a local holomorphic trivialization of $L$ on a neighborhood of $\overline{U}$, and $d\lambda$ is the $n$-dimensional Lebesgue measure on $U$.
It is obvious that $\mathcal{I}'_{\psi}(h)\supset\mathcal{I}(he^{-\psi})$.

Denote by $Y$ the zero set of the ideal sheaf
$\mathcal{J}_Y:=\mathcal{I}(he^{-\psi}):\mathcal{I}(h)$ (cf. Remark \ref{r:Y-definition}). Let $f$ be an element in
\[H^0\big(X,\mathcal{O}_X(K_X\otimes L)\otimes\mathcal{I}'_{\psi}(h)/\mathcal{I}(he^{-\psi})\big).\]
Then $f$ is actually supported on $Y$. We define a positive measure $|f|^2_{\omega,h}dV_{X,\omega}[\psi]$ (a purely formal notation, cf. \cite{Ohsawa5}, \cite{Guan-Zhou2013b} and (2.10) in \cite{Demailly2015}) on $Y$ as
the minimum element of the partially ordered set of positive measures $d\mu$
satisfying
\[\int_Ygd\mu\geq\varlimsup_{t\rightarrow-\infty}\int_{\{x\in X:\,t<\psi(x)<t+1\}}g|\tilde{f}|^2_{\omega,h}e^{-\psi}dV_{X,\omega}\]
for any nonnegative continuous function $g$ on $X$ with $\mathrm{supp}\,g\subset\subset X$, where $\tilde{f}$ is a smooth extension of $f$ to $X$ such that $\tilde{f}-\hat{f}\in\mathcal{O}_X(K_X\otimes L)\otimes_{\mathcal{O}_X}\mathcal{I}(he^{-\psi})\otimes_{\mathcal{O}_X}\mathcal{C}^\infty$ locally for any local holomorphic representation $\hat{f}$ of $f$. It is not hard to check that the upper limit on the right hand side of the above inequality is independent of the choice of $\tilde{f}$.\\

It is useful to consider $L^2$ estimates with variable factors. Let us recall the following definition in \cite{Guan-Zhou2013b}.

\begin{definition}[\cite{Guan-Zhou2013b}]
Let $\alpha_0\in(-\infty,+\infty]$ and $\alpha\in(0,+\infty]$. If $\alpha=+\infty$, $\frac{1}{\alpha}$ is defined to be $0$.
When $\alpha_0\neq+\infty$, let $\mathfrak{R}_{\alpha_0,\alpha}$ be the class of functions defined by
\begin{align*}
\bigg\{&R\in
C^{\infty}(-\infty,\alpha_0];\text{ }R>0,\text{ }R\text{ is decreasing near }-\infty,
\\&\varlimsup\limits_{t\rightarrow-\infty} e^tR(t)<+\infty,
\text{ }C_R:=\int_{-\infty}^{\alpha_0}\frac{1}{R(t)}dt<+\infty\text{ and}
\\&\int_t^{\alpha_0}\bigg(\frac{1}{\alpha R(\alpha_0)}+\int_{t_2}^{\alpha_0}\frac{dt_1}{R(t_1)}\bigg)dt_2+\frac{1}{\alpha^2R(\alpha_0)}
<R(t)\bigg(\frac{1}{\alpha R(\alpha_0)}+\int_t^{\alpha_0}\frac{dt_1}{R(t_1)}\bigg)^2
\\&\text{for all }t\in(-\infty,\alpha_0)\bigg\}.
\end{align*}
When $\alpha_0=+\infty$, we replace $R\in C^\infty(-\infty,\alpha_0]$ with the assumptions
\[R\in C^\infty(-\infty,+\infty),\quad R(+\infty):=\lim\limits_{t\rightarrow+\infty} R(t)=+\infty\quad\text{and} \quad\lim\limits_{t\rightarrow+\infty}\frac{R(t)}{R'(t)}\geq\frac{1}{\alpha}\]
in the above definition of $\mathfrak{R}_{\alpha_0,\alpha}$.\qed
\end{definition}

\begin{remark}\label{r:defintion-of-function-class}
The number $\alpha_0$, $\alpha$ and the function $R(t)$ are equal to the number $A$, $\delta$ and the function $\frac{1}{c_A(-t)e^t}$ defined just before the main theorems in \cite{Guan-Zhou2013b}. If $\alpha_0\neq+\infty$ and $R$ is decreasing on $(-\infty,\alpha_0]$, the longest inequality in the definition of $\mathfrak{R}_{\alpha_0,\alpha}$ holds for all $t\in(-\infty,\alpha_0)$.\qed
\end{remark}

\begin{thm}\label{t:Zhou-Zhu-optimal-jet-extension}
Let $\alpha_0\in(-\infty,+\infty]$, $\alpha\in(0,+\infty]$, and $R\in\mathfrak{R}_{\alpha_0,\alpha}$. Let
$X$ be a weakly pseudoconvex complex $n$-dimensional
manifold possessing a K\"{a}hler metric $\omega$, $\psi$ be an $L^1_{\mathrm{loc}}$ function on $X$ satisfying $\sup\limits_{\Omega}\psi<\alpha_0$ for every relatively compact set $\Omega\subset\subset X$, and $(L,h)$ be a holomorphic line bundle over $X$ equipped with a singular Hermitian metric $h$. Denote by $Y$ the zero set of the ideal sheaf
$\mathcal{J}_Y:=\mathcal{I}(he^{-\psi}):\mathcal{I}(h)$ (cf. Remark \ref{r:Y-definition}). Assume that the following two inequalities hold on $X$ in the sense of currents:
\\$(i)$\quad$\sqrt{-1}\Theta_{L,h}+\sqrt{-1}\partial\bar\partial\psi\geq0$,
\\$(ii)$\quad$\frac{1}{1+\alpha}\sqrt{-1}\Theta_{L,h}+\sqrt{-1}\partial\bar\partial\psi\geq0$.
\\Then for every section $f\in H^0\big(X,\mathcal{O}_X(K_X\otimes L)\otimes\mathcal{I}'_{\psi}(h)/\mathcal{I}(he^{-\psi})\big)$ such that
\begin{equation}\label{ie:jet-extension-thm-f-finite}
\int_Y|f|^2_{\omega,h} dV_{X,\omega}[\psi]<+\infty,
\end{equation}
there exists a section $F\in H^0\big(X,\mathcal{O}_X(K_X\otimes L)\otimes\mathcal{I}'_{\psi}(h)\big)$ which maps to $f$ under the morphism $\mathcal{I}'_{\psi}(h)\longrightarrow\mathcal{I}'_{\psi}(h)/\mathcal{I}(he^{-\psi})$, such that
\begin{equation}\label{ie:jet extension thm final estimate}
\int_X\frac{|F|^2_{\omega,h}}{e^{\psi}R(\psi)} dV_{X,\omega}
\leq \bigg(\frac{1}{\alpha R(\alpha_0)}+C_R\bigg)\int_Y|f|^2_{\omega,h} dV_{X,\omega}[\psi].
\end{equation}
Moreover, the restriction morphism
\[H^0\big(X,\mathcal{O}_X(K_X\otimes L)\otimes\mathcal{I}'_{\psi}(h)\big)\longrightarrow H^0\big(X,\mathcal{O}_X(K_X\otimes L)\otimes\mathcal{I}'_{\psi}(h)/\mathcal{I}(he^{-\psi})\big)\]
is surjective.
\end{thm}

\begin{remark}\label{r:Demailly-jet}
Theorem \ref{t:Zhou-Zhu-optimal-jet-extension} was proved in \cite{Demailly2015} for an explicit function $R$ with a non optimal $L^2$ estimate in the case when $\psi=(m_p-m_{p-1})\varphi$ for a quasi-psh function $\varphi$ on $X$ with neat analytic singularities, where $m_p$ are jumping numbers (see Theorem 2.13 in \cite{Demailly2015} for details).
Theorem \ref{t:Zhou-Zhu-optimal-jet-extension} gives an optimal $L^2$ estimate of Theorem 2.13 in \cite{Demailly2015}. In fact,
the constant $\frac{1}{\alpha R(\alpha_0)}+C_R$ in the $L^2$ estimate \eqref{ie:jet extension thm final estimate} is optimal since it is reached in some special cases of Theorem \ref{t:Zhou-Zhu-optimal-jet-extension} (see \cite{Guan-Zhou2013b}).\qed
\end{remark}

\begin{remark}
By using the methods in \cite{Guan-Zhou2013b} and \cite{Berndtsson-Lempert}, Hosono \cite{Hosono2017} obtained an optimal $L^2$ estimate of Theorem 2.13 in \cite{Demailly2015} in the case when $X$ is a bounded pseudoconvex domain in $\mathbb{C}^n$, $Y$ is a closed complex submanifold, $\psi$ is a negative psh Green-type function continuous on $X\setminus Y$ with poles along $Y$, and $(L,h)$ is a trivial bundle with a continuous Hermitian metric $h$.\qed
\end{remark}
\smallskip

\bigskip

%%%---------------------------------------------------

\section{Applications of Theorem \ref{t:Zhou-Zhu-nonreduced-quasipsh}}\label{section-corollaries-of-thm-1}

In this section, we will show that Theorem \ref{t:Zhou-Zhu-nonreduced-quasipsh} implies several recent injectivity theorems obtained in \cite{Matsumura2014}, \cite{Matsumura2018}, \cite{Gongyo-Matsumura2017}, \cite{Fujino-Matsumura2016} and \cite{Matsumura2016}. We will also discuss an application of Theorem \ref{t:Zhou-Zhu-nonreduced-quasipsh} to vanishing theorems.

By Theorem \ref{t:Zhou-Zhu-nonreduced-quasipsh}, we can obtain the following result, which unifies some well-known injectivity theorems.

\begin{thm}\label{cor:FM2016}
Let $X$ be a holomorphically convex K\"{a}hler manifold. Let $(F,h_F)$ and $(M,h_M)$ be two holomorphic line bundles over $X$ equipped with singular Hermitian metrics $h_F$ and $h_M$ respectively. Assume that the following two inequalities hold on $X$ in the sense of currents:
\\$(i)$\quad$\sqrt{-1}\Theta_{F,h_F}\geq0$,
\\$(ii)$\quad$\sqrt{-1}\Theta_{F,h_F}\geq b\sqrt{-1}\Theta_{M,h_M}$ for some $b\in(0,+\infty)$.
\\Then, for a non-zero global holomorphic section $s$ of $M$ satisfying $\sup\limits_{\Omega}|s|_{h_M}<+\infty$ for every relatively compact set $\Omega\subset\subset X$, the following map $\beta$ induced by the tensor product with $s$
\[H^q\big(X,\mathcal{O}_X(K_X\otimes F)\otimes\mathcal{I}(h_F)\big)\overset{\beta}\longrightarrow H^q\big(X,\mathcal{O}_X(K_X\otimes F\otimes M)\otimes \mathcal{I}(h_Fh_M)\big)\]
is injective for every $q\geq0$.
\end{thm}

\begin{proof}
Let $\psi:=2\log|s|_{h_M}$. Then $\psi$ is an $L^1_{\mathrm{loc}}$ function on $X$ which is locally bounded above.

Let $\mathcal{U}=\{U_i\}_{i\in I}$ be a Stein covering of $X$.

It is easy to see that the following maps induced by the tensor product with $s$
\[H^0\big(U_i,\mathcal{O}_X(K_X\otimes F)\otimes\mathcal{I}(h_F)\big)\longrightarrow H^0\big(U_i,\mathcal{O}_X(K_X\otimes F\otimes M)\otimes\mathcal{I}(h_Fh_Me^{-\psi})\big)\]
are isomorphisms. Hence they induce isomorphisms between the spaces of \v{C}ech cochains $C^p\big(\mathcal{U},\mathcal{O}_X(K_X\otimes F)\otimes\mathcal{I}(h_F)\big)$ and $C^p\big(\mathcal{U},\mathcal{O}_X(K_X\otimes F\otimes M)\otimes\mathcal{I}(h_Fh_Me^{-\psi})\big)$.

Since these isomorphisms between the spaces of \v{C}ech cochains commute with the \v{C}ech coboundary mappings, it follows from Leray's theorem that the following map $\sigma$ induced by the tensor product with $s$
\[H^q\big(X,\mathcal{O}_X(K_X\otimes F)\otimes\mathcal{I}(h_F)\big)\overset{\sigma}\longrightarrow H^q\big(X,\mathcal{O}_X(K_X\otimes F\otimes M)\otimes\mathcal{I}(h_Fh_Me^{-\psi})\big)\]
is an isomorphism.

Let $\iota$ be the following map induced by the natural inclusion $\mathcal{I}(h_Fh_Me^{-\psi})\longrightarrow\mathcal{I}(h_Fh_M)$
\[H^q\big(X,\mathcal{O}_X(K_X\otimes F\otimes M)\otimes\mathcal{I}(h_Fh_Me^{-\psi})\big)\overset{\iota}\longrightarrow H^q\big(X,\mathcal{O}_X(K_X\otimes F\otimes M)\otimes \mathcal{I}(h_Fh_M)\big).\]
Then we get the injectivity of $\iota$ by applying Theorem \ref{t:Zhou-Zhu-nonreduced-quasipsh} to the case when $L:=F\otimes M$, $h_L:=h_Fh_M$ and $\alpha:=b$.

Therefore, the map $\beta$ in Theorem \ref{cor:FM2016} is injective by the relation $\beta=\iota\circ\sigma$. Hence we get Theorem \ref{cor:FM2016}.
\end{proof}
\smallskip

\begin{remark}
Theorem \ref{cor:FM2016} was proved in \cite{Matsumura2014} in the case when $X$ is compact, and the metrics $h_F$, $h_M$ are both smooth on some Zariski open subset of $X$ (see Theorem 1.5 in \cite{Matsumura2014}).\qed
\end{remark}

\begin{remark}
Theorem \ref{cor:FM2016} unifies the two main results Theorem 1.2 and Theorem 1.3 in \cite{Matsumura2016}. More precisely, Theorem \ref{cor:FM2016} was proved in \cite{Matsumura2016} under one of the following two additional assumptions:
\\$(1)$\quad$(M,h_M)=(F^m,h_F^m)$ for some nonnegative integer $m$ (see Theorem 1.2 in \cite{Matsumura2016}, and see also Theorem 1.3 in \cite{Matsumura2018} for the case when $X$ is compact);
\\$(2)$\quad$h_M$ is smooth and $\sqrt{-1}\Theta_{M,h_M}\geq0$ (see Theorem 1.3 in \cite{Matsumura2016}, and see also Theorem A in \cite{Fujino-Matsumura2016} for the case when $X$ is compact).\qed
\end{remark}

\begin{remark}
Theorem \ref{cor:FM2016} was also proved in \cite{Gongyo-Matsumura2017} in the case when $X$ is compact, $\sqrt{-1}\Theta_{M,h_M}\geq0$ and $h_F=h^b_Mh_\Delta$ for some $b\in(0,+\infty)$ and some effective $\mathbb{R}$-divisor $\Delta$ on $X$, where $h_\Delta$ is the singular Hermitian metric defined by $\Delta$ (see Theorem 1.3 in \cite{Gongyo-Matsumura2017}). Thus the dlt extension theorem 1.4 in \cite{Gongyo-Matsumura2017} for compact K\"{a}hler manifolds also holds for holomorphic convex K\"{a}hler manifolds by using Theorem \ref{cor:FM2016} and the same arguments as in \cite{Gongyo-Matsumura2017} (see also \cite{Siu98}, \cite{Siu02}, \cite{Siu04} and \cite{Demailly-Hacon-Paun2013} for the background of the dlt extension problem).\qed
\end{remark}
\medskip

Injectivity theorems have benn used to obtain vanishing theorems in many previous important works. Using Theorem \ref{cor:FM2016}, we can get a vanishing theorem (Theorem \ref{cor:vanishing-thm}). Before stating the result, let's recall some notions and notations.

For any holomorphic line bundle $(L,h_L)$ equipped with a singular Hermitian metric $h_L$ over a compact complex manifold $X$, denote by $H^0_{\mathrm{bdd},h_L}(X,L)$ the space of the holomorphic sections of $L$ with bounded norms. Namely,
\[H^0_{\mathrm{bdd},h_L}(X,L):=\{s\in H^0(X,L):\,\sup_{X}|s|_{h_L}<+\infty\}.\]

Let $\{h_k\}_{k=1}^{+\infty}$ be a sequence of singular Hermitian metrics on $L$. The {\bf generalized Kodaira dimension} $\kappa_{\mathrm{bdd}}(L,\{h_k\}_{k=1}^{+\infty})$ of $(L,\{h_k\}_{k=1}^{+\infty})$ is defined to be $-\infty$ if $H^0_{\mathrm{bdd},(h_k)^k}(X,L^k)=0$ for any positive integer $k$ which is large enough. Otherwise, $\kappa_{\mathrm{bdd}}(L,\{h_k\}_{k=1}^{+\infty})$ is defined to be
\[\sup\bigg\{m\in \mathbb{Z}:\,\varlimsup_{k\rightarrow+\infty}\frac{\dim H^0_{\mathrm{bdd},(h_k)^k}(X,L^k)}{k^m}>0\bigg\}.\]

For more details about the generalized Kodaira dimension in the case when $h_k=h_L$ $(\forall k\in \mathbb{Z}^+)$ for some fixed metric $h_L$ on $L$, one can see Section 5.2 in \cite{Matsumura2014} and Section 4 in \cite{Matsumura2018}.

\begin{thm}\label{cor:vanishing-thm}
Let $X$ be a complex $n$-dimensional projective manifold, and $(F,h_F)$ be a holomorphic line bundle over $X$ equipped with a singular Hermitian metric $h_F$ satisfying $\sqrt{-1}\Theta_{F,h_F}\geq0$ on $X$ in the sense of currents. Let $Q$ be a holomorphic line bundle over $X$, and $\{h_k\}_{k=1}^{+\infty}$ be a sequence of singular Hermitian metrics on $Q$. Assume that the following two inequalities hold on $X$ in the sense of currents:
\\$(i)$\quad$\sqrt{-1}\Theta_{F,h_F}\geq\varepsilon_k\sqrt{-1}\Theta_{Q,h_k}$ $(\forall k\in \mathbb{Z}^+)$ for some sequence of positive numbers $\{\varepsilon_k\}_{k=1}^{+\infty}$,
\\$(ii)$\quad$\sqrt{-1}\Theta_{F,h_F}+k\sqrt{-1}\Theta_{Q,h_k}\geq-C\omega$ $(\forall k\in \mathbb{Z}^+)$ for some positive number $C$ and some smooth positive $(1,1)$-form $\omega$ on $X$.
\\Then
\[H^q\big(X,\mathcal{O}_X(K_X\otimes F)\otimes\mathcal{I}(h_F)\big)=0\quad\text{for}\quad q>n-\kappa_{\mathrm{bdd}}(Q,\{h_k\}_{k=1}^{+\infty}),\]
where $\kappa_{\mathrm{bdd}}(Q,\{h_k\}_{k=1}^{+\infty})$ is the generalized Kodaira dimension of $(Q,\{h_k\}_{k=1}^{+\infty})$.
\end{thm}

\begin{proof}
The proof is similar to that of Theorem 1.4 (2) in \cite{Matsumura2014}.

Theorem \ref{cor:vanishing-thm} holds trivially if $\kappa_{\mathrm{bdd}}(Q,\{h_k\}_{k=1}^{+\infty})\leq0$. Hence we assume that $\kappa_{\mathrm{bdd}}(Q,\{h_k\}_{k=1}^{+\infty})$ is positive.

For a contradiction, we assume that there exists a non-zero cohomology class $\xi\in H^q\big(X,\mathcal{O}_X(K_X\otimes F)\otimes\mathcal{I}(h_F)\big)$ for some $q>n-\kappa_{\mathrm{bdd}}(Q,\{h_k\}_{k=1}^{+\infty})$.

Let $k$ be a positive integer. Then, the following map induced by the tensor product with $\xi$
\[H^0_{\mathrm{bdd},(h_k)^k}(X,Q^k)\longrightarrow H^q\big(X,\mathcal{O}_X(K_X\otimes F\otimes Q^k)\otimes\mathcal{I}(h_F(h_k)^k)\big)\]
is a linear map.

Since $\sqrt{-1}\Theta_{F,h_F}\geq0$ and $\sqrt{-1}\Theta_{F,h_F}\geq\varepsilon_k\sqrt{-1}\Theta_{Q,h_k}$ hold on $X$, applying Theorem \ref{cor:FM2016} to the above linear map in the case when $(M,h_M):=(Q^k,(h_k)^k)$ and $b:=\frac{\varepsilon_k}{k}$, we get that the above linear map is injective by the assumption $\xi\neq0$.

Hence
\[\dim H^0_{\mathrm{bdd},(h_k)^k}(X,Q^k)\leq\dim H^q\big(X,\mathcal{O}_X(K_X\otimes F\otimes Q^k)\otimes\mathcal{I}(h_F(h_k)^k)\big).\]

By Nadel's vanishing theorem (see \cite{Nadel1990} or Theorem 5.11 in \cite{Demailly-book2010}) and the assumption $(ii)$ in Theorem \ref{cor:vanishing-thm}, it follows from Lemma \ref{l:asymptotic estimate} that
\[\dim H^q\big(X,\mathcal{O}_X(K_X\otimes F\otimes Q^k)\otimes\mathcal{I}(h_F(h_k)^k)\big)=O(k^{n-q})\quad\text{as}\quad k\rightarrow+\infty.\]
Hence
\[\dim H^0_{\mathrm{bdd},(h_k)^k}(X,Q^k)=O(k^{n-q})\quad\text{as}\quad k\rightarrow+\infty.\]

By the definition of the generalized Kodaira dimension, the above equality is a contradiction to the inequality $q>n-\kappa_{\mathrm{bdd}}(Q,\{h_k\}_{k=1}^{+\infty})$. Hence we get Theorem \ref{cor:vanishing-thm}.
\end{proof}
\smallskip

\begin{remark}
Theorem \ref{cor:vanishing-thm} contains the following two cases.

The first case is when $h_k=h_Q$ $(\forall k\in \mathbb{Z}^+)$ for some fixed singular metric $h_Q$ on $Q$. Then the curvature assumptions $(i)$ and $(ii)$ are equivalent to the assumption
\[\sqrt{-1}\Theta_{F,h_F}\geq\varepsilon\sqrt{-1}\Theta_{Q,h_Q}\geq0\text{ for some positive number }\varepsilon.\]
If $h_Q$ is further a smooth metric with strictly positive curvature, then Theorem \ref{cor:vanishing-thm} is just Nadel's vanishing theorem.

The second case is when $Q$ is numerically effective. Then there exists a sequence of smooth metrics $\{h_k\}_{k=1}^{+\infty}$ such that the curvature assumption $(ii)$ holds. In this case, $\kappa_{\mathrm{bdd}}(Q,\{h_k\}_{k=1}^{+\infty})$ is just the usual Kodaira dimension of $Q$.\qed
\end{remark}

\begin{remark}
In the case when $\varepsilon_k=1$ and $h_k=h_Q$ $(\forall k\in \mathbb{Z}^+)$ for some fixed smooth metric $h_Q$ on $Q$, Theorem \ref{cor:vanishing-thm} was proved in \cite{Matsumura2014} under the additional assumption that $h_F$ is smooth on some Zariski open subset of $X$ (see Theorem 1.4 (2) in \cite{Matsumura2014}).\qed
\end{remark}

\begin{remark}
Theorem \ref{cor:vanishing-thm} was proved in \cite{Matsumura2018} in the case when $Q=F$ and $h_k=h_F$ $(\forall k\in \mathbb{Z}^+)$ (see Theorem 4.5 in \cite{Matsumura2018}, and see also Theorem 1.4 (1) in \cite{Matsumura2014}, Theorem 5.2 in \cite{Matsumura2014}, Theorem 1.2 in \cite{Matsumura2015}).\qed
\end{remark}

\begin{remark}
By the vanishing theorem obtained in \cite{Cao2012} and the strong openness property of multiplier ideal sheaves obtained in \cite{Guan-Zhou2013c}, one can get
\[H^q\big(X,\mathcal{O}_X(K_X\otimes F)\otimes\mathcal{I}(h_F)\big)=0\quad\text{for}\quad q>n-\mathrm{nd}(F,h_F).\]
Here $\mathrm{nd}(F,h_F)$ is the numerical dimension defined in \cite{Tsuji2007} and \cite{Cao2012}, which depends on the pair $(F, h_F)$.
Here, $\kappa_{\mathrm{bdd}}(Q,\{h_k\}_{k=1}^{+\infty})$ rather than $\mathrm{nd}(F,h_F)$ is used in Theorem \ref{cor:vanishing-thm}. \qed
\end{remark}
\bigskip

%%%-----------------------------------------------------

\section{Some results used in the proofs}\label{section-lemmas}

In this section, we recall and obtain some results which will be used in the
proofs of the main results in the present paper.

\begin{lem}[Proposition 3.12 in \cite{Demailly2015}]\label{l:Demailly-non complete metric}
Let $X$ be a complete K\"{a}hler manifold equipped with a
(non necessarily complete) K\"{a}hler metric $\omega$, and let $(Q,h)$
be a holomorphic vector bundle over $ X $ equipped with a smooth Hermitian metric $h$. Assume that $\tau$ and $A$
are smooth and bounded positive functions on $ X $ and let
\[\mathrm{B}:=[\tau\sqrt{-1}\Theta_{Q,h}-\sqrt{-1}
\partial\bar\partial\tau
-\sqrt{-1}A^{-1}\partial\tau\wedge\bar\partial\tau,\Lambda ].\]
Assume that $\delta\geq0$ is a number such that $\mathrm{B}+\delta
\mathrm{I}$ is semi-positive definite everywhere on
$\wedge^{n,q}T^*_ X \otimes Q$ for some $q\geq 1$. Then given a form
$g\in L^2( X ,\wedge^{n,q}T^*_ X \otimes Q)$ such that $\bar\partial
g=0$ and
\[\int_ X \langle {(\mathrm{B}+\delta
\mathrm{I})}^{-1}g,g\rangle_{\omega,h} dV_{X,\omega}<+\infty,\]
there exists an
approximate solution $u\in L^2( X ,\wedge^{n,q-1}T^*_ X \otimes Q)$
and a correcting term $v\in L^2( X ,\wedge^{n,q}T^*_ X \otimes Q)$
such that $\bar\partial  u+\sqrt{\delta}v=g$ and
\[\int_ X \frac{|u|^2_{\omega,h}}{\tau+A} dV_{X,\omega}+\int_X|v|^2_{\omega,h} dV_{X,\omega}\leq \int_ X
\langle {(\mathrm{B}+\delta \mathrm{I})}^{-1}g,g\rangle_{\omega,h} dV_{X,\omega}.\]
\end{lem}

\begin{lem}[Theorem 4.4.2 in \cite{Hormander}]\label{l:Hormander-dbar}
Let $\Omega$ be a pseudoconvex
open set in $\mathbb{C}^n$, and $\varphi$ be a plurisubharmonic
function on $\Omega$. For every $w\in L^2_{(p,q+1)}(\Omega,e^{-\varphi})$
with $\bar\partial w=0$ there is a solution $s\in
L^2_{(p,q)}(\Omega,\mathrm{loc})$ of the equation $\bar\partial s=w$
such that
\[\int_\Omega\frac{|s|^2}{(1+|z|^2)^2}e^{-\varphi}d\lambda\leq
\int_\Omega|w|^2e^{-\varphi}d\lambda,\]
where $d\lambda$ is the $2n$-dimensional Lebesgue measure on
$\mathbb{C}^n$.
\end{lem}

\begin{lem}[Theorem 1.5 in \cite{Demailly82}]\label{l:Demailly1982-complete-metric}
Let $X$ be a K\"{a}hler
manifold, and $Z$ be an analytic subset of $X$. Assume that $\Omega$
is a relatively compact open subset of $X$ possessing a complete
K\"{a}hler metric. Then $\Omega\setminus Z$ carries a complete
K\"{a}hler metric.
\end{lem}

\begin{lem}[Lemma 6.9 in \cite{Demailly82}]\label{l:extension}
Let $\Omega$ be an open subset of
$\mathbb{C}^n$ and $Z$ be a complex analytic subset of $\Omega$.
Assume that $u$ is a $(p,q-1)$-form with $L^2_{\mathrm{loc}}$
coefficients and $g$ is a $(p,q)$-form with $L^1_{\mathrm{loc}}$
coefficients such that $\bar\partial u=g$ on $\Omega\setminus Z$ (in the sense of currents). Then $\bar\partial u=g$ on
$\Omega$.
\end{lem}

\begin{lem}[The open mapping theorem, cf. \cite{Rudin1991}]\label{l:open mapping theorem}
Let $\mathrm{T}:F_1\longrightarrow F_2$ be a linear map between
Fr\'{e}chet spaces $F_1$ and $F_2$. If $\mathrm{T}$ is continuous and surjective, then $\mathrm{T}$ is open.
\end{lem}

\begin{lem}[Theorem 2 of Section D in Chapter II of \cite{Gunning-Rossi}]\label{l:local generators uniform estimate}
Let $U$ be an open neighborhood of the origin $0$ in $\mathbb{C}^n$, and let $G_1,\cdots,G_k$ be holomorphic functions on $U$. Denote by $\mathcal{O}_K$ the ring of germs of holomorphic functions on the set $K$ for any closed set $K\subset \mathbb{C}^n$, and denote by $\mathcal{G}$ the ideal of $\mathcal{O}_0$ generated by the germs of $G_1,\cdots,G_k$ at $0$. Then
there exists an open neighborhood $V\subset\subset U$ of $0$ and a positive number $C$, such that every $F\in\mathcal{O}_{\overline{V}}$ whose germ at $0$ belongs to $\mathcal{G}$ can be written in the form
\[F=\sum\limits_{j=1}^k a_jG_j\text{ as germs on }\overline{V},\]
where $a_j\in\mathcal{O}_{\overline{V}}$ and $\sup\limits_{\overline{V}}|a_j|\leq C\sup\limits_{\overline{V}}|F|$.
\end{lem}

\begin{lem}[Lemma 4.3 in \cite{Matsumura2014}]\label{l:asymptotic estimate}
Let $X$ be a complex $n$-dimensional projective manifold, and $Q$ be a holomorphic line bundle on $X$. Let $\mathcal{G}$ be a coherent analytic sheaf on $X$, and $\{\mathcal{I}_k\}_{k=1}^{+\infty}$ be ideal sheaves on $X$ such that there exists a very ample line bundle $A$ on $X$ satisfying
\[H^q(X,A^m\otimes\mathcal{G}\otimes Q^k\otimes\mathcal{I}_k)=0 \text{ for any positive integers } q,\,m,\,k.\]
Then for any $q\geq0$, we have
\[\dim H^q(X,\mathcal{G}\otimes Q^k\otimes\mathcal{I}_k)=O(k^{n-q})\quad\text{as}\quad k\rightarrow+\infty.\]
\end{lem}

\begin{thm}[Theorem 6.1 in \cite{Demailly94}]\label{l:Demailly1994}
Let $X$ be a complex manifold equipped with a Hermitian
metric $\omega$, and $\Omega\subset\subset X$ be an open subset. Suppose that the Chern
curvature tensor of $T_X$ satisfies
\[\bigg(\frac{\sqrt{-1}}{2\pi}\Theta_{T_X}+\varpi\otimes\mathrm{Id}_{T_X}\bigg)
(\kappa_1\otimes\kappa_2,\kappa_1\otimes\kappa_2)
\geq0\quad(\forall\kappa_1,\kappa_2\in T_X\text{ with }\langle
\kappa_1,\kappa_2\rangle=0)\]
on a neighborhood of $\overline{\Omega}$, for some continuous nonnegative
$(1,1)$-form $\varpi$ on $X$. Assume that $\varphi$ is a quasi-psh function on $X$, and let $\gamma$ be a continuous real
$(1,1)$-form such that $\sqrt{-1}\partial\bar\partial\varphi\geq\gamma$ in the sense of currents. Then there is a family of functions $\varphi_{\varsigma,\rho}$ defined
on a neighborhood of $\overline{\Omega}$ ($\varsigma\in(0,+\infty)$ and $\rho\in(0,\rho_1)$ for some
positive number $\rho_1$) independent of $\gamma$, such that
\begin{enumerate}
\item[$(i)$]
$\varphi_{\varsigma,\rho}$ is quasi-psh on a neighborhood of $\overline{\Omega}$,
smooth on $\overline{\Omega}\setminus E_\varsigma(\varphi)$, increasing with respect to
$\varsigma$ and $\rho$ on $\overline{\Omega}$, and converges to $\varphi$ on
$\overline{\Omega}$ as $\rho\rightarrow0$,
\item[$(ii)$]
$\sqrt{-1}\partial\bar\partial\varphi_{\varsigma,\rho}\geq\gamma-\pi\varsigma \varpi-\delta_\rho\omega$ on $\Omega$,
\end{enumerate}
where $E_\varsigma(\varphi):=\{x\in X:\,\nu(\varphi,x)\geq \varsigma\}$
$(\varsigma>0)$ is the $\varsigma$-upperlevel set of Lelong numbers,
and $\{\delta_\rho\}$ is an increasing family of positive numbers
such that $\lim\limits_{\rho\rightarrow 0}\delta_\rho=0$.
\end{thm}

\begin{remark}
Although Theorem \ref{l:Demailly1994} is stated in \cite{Demailly94} in the case $X$ is compact,
almost the same proof as in \cite{Demailly94} shows that Theorem \ref{l:Demailly1994} holds in the noncompact case while uniform
estimates are obtained only on the relatively compact subset
$\Omega$.\qed
\end{remark}

We need to use the proof of Theorem \ref{l:Demailly1994} to obtain an approximation lemma below. Let us first review the construction of $\varphi_{\varsigma,\rho}$ in \cite{Demailly94}.

Select a smooth cut-off function $\theta:\mathbb{R}\longrightarrow \mathbb{R}$ such that
\[\theta(t)>0\text{ for }t<1,\text{ }\theta(t)=0\text{ for }t\geq1,\text{ and }
\int_{v\in\mathbb{C}^n}\theta(|v|^2)d\lambda(v)=1.\]
We set
\[\varphi_\rho(x)=\frac{1}{\rho^{2n}}\int_{\{\zeta\in T_{X,x}:\,|\zeta|<\rho\}}\varphi\big(\mathrm{exph}_x(\zeta)\big)\theta\bigg(\frac{|\zeta|^2}{\rho^2}\bigg)d\lambda(\zeta),\quad \rho>0,\]
where $\mathrm{exph}_x(\zeta)$ is a smooth map modified from the exponential map, and $d\lambda(\zeta)$ denotes the Lebesgue measure on the Hermitian space $\big(T_{X,x},\omega(x)\big)$.

Let $\rho$ be a small positive number. For $w\in \mathbb{C}$ with $|w|=\rho$, we have
\[\varphi_\rho(x)=\Phi(x,w)=\Phi(x,\rho)\]
with
\begin{equation}\label{e:def-Phi}
\Phi(x,w):=\int_{\{\zeta\in T_{X,x}:\,|\zeta|<1\}}\varphi\big(\mathrm{exph}_x(w\zeta)\big)\theta(|\zeta|^2)d\lambda(\zeta).
\end{equation}

In \cite{Demailly94}, Demailly proved that there is a positive number $K$ and a positive number $\rho_0$ such that
\begin{enumerate}
\item[$(a)$]
$\Phi(x,w)$ is smooth over $\overline{\Omega}\times\{0<|w|<\rho_0\}$ and $\lim\limits_{\rho\rightarrow 0}\Phi(x,\rho)=\varphi(x)$ for any $x\in\overline{\Omega}$,
\item[$(b)$]
$\Phi(x,\rho)+K\rho^2$ is convex and increasing in $\log\rho$ when $\rho\in(0,\rho_0)$ and $x\in\overline{\Omega}$,
\item[$(c)$]
$\Phi(x,\rho)$ is quasi-psh on a neighborhood of $\overline{\Omega}$ for any fixed $\rho\in(0,\rho_0)$.
\end{enumerate}

By following an idea of Kiselman \cite{Kiselman1979}, $\varphi_{\varsigma,\rho}$ was defined in \cite{Demailly94} to be the Legendre transform
\[\varphi_{\varsigma,\rho}(x)=\inf\limits_{0<|w|<1}\bigg(\Phi(x,\rho|w|)+\rho|w|+\frac{\rho}{1-|w|^2}-\varsigma\log|w|\bigg).\]
Then it is clear that $\varphi_{\varsigma,\rho}$ is increasing in $\varsigma$ and $\rho$, and that
\[\lim\limits_{\rho\rightarrow 0}\varphi_{\varsigma,\rho}(x)=\lim\limits_{|w|\rightarrow 0}\Phi(x,w)=\varphi(x).\]

By a lot of computations, Demailly obtained estimates for $\sqrt{-1}\partial\bar\partial\Phi$ and proved that $\varphi_{\varsigma,\rho}$ satisfies the conclusion of Theorem \ref{l:Demailly1994}.

Now, by using the construction of $\varphi_{\varsigma,\rho}$, we prove the following approximation lemma, which plays an important role in the proofs of Theorem \ref{t:Zhou-Zhu-nonreduced-quasipsh} and Theorem \ref{t:Zhou-Zhu-optimal-jet-extension}.

\begin{lem}\label{l:regularization-two-quasipsh}
Let $X$ be a complex manifold equipped with a Hermitian
metric $\omega$, and $\Omega\subset\subset X$ be an open subset.
Suppose that the Chern
curvature tensor of $T_X$ satisfies
\[\bigg(\frac{\sqrt{-1}}{2\pi}\Theta_{T_X}+\varpi\otimes\mathrm{Id}_{T_X}\bigg)
(\kappa_1\otimes\kappa_2,\kappa_1\otimes\kappa_2)
\geq0\quad(\forall\kappa_1,\kappa_2\in T_X\text{ with }\langle
\kappa_1,\kappa_2\rangle=0)\]
on a neighborhood of $\overline{\Omega}$, for some continuous nonnegative
$(1,1)$-form $\varpi$ on $X$.
Let $\varphi_1$ be a quasi-psh function on $X$, and $\varphi_2$ be an $L^1_{\mathrm{loc}}$ function on $X$ which is bounded above. Assume that $\varphi_1+\varphi_2$ and $\varphi_1+(1+\alpha)\varphi_2$ are quasi-psh on $X$ such that the following two inequalities hold on $X$ in the sense of currents:
\[\sqrt{-1}\partial\bar\partial\varphi_1+\sqrt{-1}\partial\bar\partial\varphi_2\geq\gamma_1,\]
\[\sqrt{-1}\partial\bar\partial\varphi_1+(1+\alpha)\sqrt{-1}\partial\bar\partial\varphi_2\geq\gamma_2,\]
where $\gamma_1$ and $\gamma_2$ are continuous real $(1,1)$-forms on $X$, and $\alpha$ is a positive number.
Let
\[\Sigma:=\{\varphi_1+\varphi_2=-\infty\}\cup\{\varphi_1+(1+\alpha)\varphi_2=-\infty\}\]
and
\[\Sigma_\varsigma:=E_\varsigma\big(\varphi_1+\varphi_2\big)\cup E_\varsigma\big(\varphi_1+(1+\alpha)\varphi_2\big),\]
where $E_\varsigma(\varphi):=\{x\in X:\,\nu(\varphi,x)\geq \varsigma\}$
$(\varsigma>0)$ is the $\varsigma$-upperlevel set of Lelong numbers for a quasi-psh function $\varphi$ on $X$.
Then there are two family of upper semicontinuous functions $\{\varphi_{1,\varsigma,\rho}\}$ and $\{\varphi_{2,\varsigma,\rho}\}$ defined on a neighborhood of $\overline{\Omega}$ with values in $[-\infty,+\infty]$ ($\varsigma\in(0,+\infty)$ and $\rho\in(0,\rho_1)$ for some positive number $\rho_1$) independent of $\gamma_1$ and $\gamma_2$, such that
\begin{enumerate}
\item[$(i)$]
$\varphi_{1,\varsigma,\rho}+\varphi_{2,\varsigma,\rho}$ and $\varphi_{1,\varsigma,\rho}+(1+\alpha)\varphi_{2,\varsigma,\rho}$ are quasi-psh on a neighborhood of $\overline{\Omega}$ up to their values on $\Sigma_\varsigma$,
smooth on $\overline{\Omega}\setminus \Sigma_\varsigma$, increasing with respect to
$\varsigma$ on $\overline{\Omega}\setminus \Sigma$, increasing with respect to $\rho$ on $\overline{\Omega}\setminus \Sigma_\varsigma$, and converge to $\varphi_1+\varphi_2$ and $\varphi_1+(1+\alpha)\varphi_2$ on
$\overline{\Omega}\setminus \Sigma_\varsigma$ respectively as $\rho\rightarrow0$ (we say that a function $f$ is quasi-psh on an open set $\Omega_1$ up to its values on a set $\Sigma_1$ if $f=g$ on $\Omega_1\setminus\Sigma_1$ for some quasi-psh function $g$ on $\Omega_1$),
\item[$(ii)$]
$\varphi_{1,\varsigma,\rho}\geq\varphi_1$ on $\Omega$, $\varphi_{2,\varsigma,\rho}\leq\sup\limits_{X}\varphi_2$ on $\Omega$, and $\bar\partial\varphi_{k,\varsigma,\rho}\in L^1$ on $\Omega$ for $k=1,2$,
\item[$(iii)$]
$\sqrt{-1}\partial\bar\partial\varphi_{1,\varsigma,\rho}+\sqrt{-1}\partial\bar\partial\varphi_{2,\varsigma,\rho}
\geq\gamma_1-\pi\varsigma\varpi-\delta_\rho\omega$ on $\Omega$,
\item[$(iv)$]
$\sqrt{-1}\partial\bar\partial\varphi_{1,\varsigma,\rho}+(1+\alpha)\sqrt{-1}\partial\bar\partial\varphi_{2,\varsigma,\rho}
\geq\gamma_2-\pi\varsigma\varpi-\delta_\rho\omega$ on $\Omega$,
\end{enumerate}
where $\{\delta_\rho\}$ is an increasing family of positive numbers
such that $\lim\limits_{\rho\rightarrow 0}\delta_\rho=0$.
\end{lem}

\begin{proof}
As in \eqref{e:def-Phi}, we set
\[\Phi_k(x,w)=\int_{\{\zeta\in T_{X,x}:\,|\zeta|<1\}}
\varphi_k\big(\mathrm{exph}_x(w\zeta)\big)\theta(|\zeta|^2)d\lambda(\zeta),\quad k=1,2.\]
Let
\[\Upsilon_{1,\varsigma,\rho}(x):=\inf\limits_{0<|w|<1}\bigg(\Phi_1(x,\rho|w|)
+\Phi_2(x,\rho|w|)+\rho|w|+\frac{\rho}{1-|w|^2}-\varsigma\log|w|\bigg)\]
and
\[\Upsilon_{2,\varsigma,\rho}(x):=\inf\limits_{0<|w|<1}\bigg(\Phi_1(x,\rho|w|)
+(1+\alpha)\Phi_2(x,\rho|w|)+\rho|w|+\frac{\rho}{1-|w|^2}-\varsigma\log|w|\bigg).\]

If $x\notin\Sigma_\varsigma$, define
\[\varphi_{1,\varsigma,\rho}(x)=\frac{1+\alpha}{\alpha}\Upsilon_{1,\varsigma,\rho}(x)-\frac{1}{\alpha}\Upsilon_{2,\varsigma,\rho}(x)\]
and
\[\varphi_{2,\varsigma,\rho}(x)=\frac{1}{\alpha}\Upsilon_{2,\varsigma,\rho}(x)-\frac{1}{\alpha}\Upsilon_{1,\varsigma,\rho}(x).\]

If $x\in\Sigma_\varsigma$, define
\[\varphi_{1,\varsigma,\rho}(x)=\varlimsup\limits_{\Sigma_\varsigma\not\ni y\rightarrow x}\varphi_{1,\varsigma,\rho}(y)\quad\text{and}\quad\varphi_{2,\varsigma,\rho}(x)=\varlimsup\limits_{\Sigma_\varsigma\not\ni y\rightarrow x}\varphi_{2,\varsigma,\rho}(y).\]

Hence we have
\[\varphi_{1,\varsigma,\rho}(x)+\varphi_{2,\varsigma,\rho}(x)=\Upsilon_{1,\varsigma,\rho}(x)\]
and
\[\varphi_{1,\varsigma,\rho}(x)+(1+\alpha)\varphi_{2,\varsigma,\rho}(x)=\Upsilon_{2,\varsigma,\rho}(x)\]
for $x\notin\Sigma_\varsigma$.

Therefore, $(i)$, $(iii)$ and $(iv)$ holds by Theorem \ref{l:Demailly1994}. It is also easy to see that $\varphi_{1,\varsigma,\rho}$ and $\varphi_{2,\varsigma,\rho}$ are upper semicontinuous.

Let $A:(0,1)\longrightarrow\mathbb{R}$ and $B:(0,1)\longrightarrow\mathbb{R}$ be two functions such that
\[\inf\limits_{0<t<1}A(t)>-\infty.\]

The simple property
\[\inf\limits_{0<t<1}\big(A(t)+B(t)\big)-\inf\limits_{0<t<1}A(t)\geq\inf\limits_{0<t<1}B(t)\]
implies that
\[\varphi_{1,\varsigma,\rho}(x)\geq\inf\limits_{0<|w|<1}\bigg(\Phi_1(x,\rho|w|)+\rho|w|+\frac{\rho}{1-|w|^2}-\varsigma\log|w|\bigg)\]
for any $x\in\overline{\Omega}\setminus\Sigma_\varsigma$.

Hence $\varphi_{1,\varsigma,\rho}\geq\varphi_1$ on $\Omega\setminus\Sigma_\varsigma$ when $\rho$ is small enough by the conclusion $(i)$ in Theorem \ref{l:Demailly1994}. Then $\varphi_{1,\varsigma,\rho}\geq\varphi_1$ on $\Omega$ by the quasi-plurisubharmonicity of $\varphi_1$ and the definition of $\varphi_{1,\varsigma,\rho}$ on $\Sigma_\varsigma$.

The simple property
\[\inf\limits_{0<t<1}\big(A(t)+B(t)\big)-\inf\limits_{0<t<1}A(t)\leq\sup\limits_{0<t<1}B(t)\]
implies that
\[\varphi_{2,\varsigma,\rho}(x)\leq\sup\limits_{0<|w|<1}\Phi_2(x,\rho|w|)\]
for any $x\in\overline{\Omega}\setminus\Sigma_\varsigma$.

Since it is easy to see that $\Phi_2(x,w)\leq\sup\limits_{X}\varphi_2$  for any $x\in\overline{\Omega}$ when $|w|$ is small enough, we get $\varphi_{2,\varsigma,\rho}(x)\leq\sup\limits_{X}\varphi_2$ for any $x\in\overline{\Omega}\setminus\Sigma_\varsigma$ when $\rho$ is small enough. Hence $\varphi_{2,\varsigma,\rho}\leq\sup\limits_{X}\varphi_2$ on $\overline{\Omega}$ when $\rho$ is small enough by the definition of $\varphi_{2,\varsigma,\rho}$ on $\Sigma_\varsigma$.

Since $\Upsilon_{1,\varsigma,\rho}$ and $\Upsilon_{2,\varsigma,\rho}$ are quasi-psh functions on a neighborhood of $\overline{\Omega}$ by Theorem \ref{l:Demailly1994}, and the first partial derivatives of any quasi-psh function are in $L^1_{\mathrm{loc}}$, we get that $\bar\partial\varphi_{k,\varsigma,\rho}\in L^1$ on $\Omega$ for $k=1,2$.

Therefore, we get $(ii)$.
\end{proof}

Let $X$ be an $n$-dimensional complex analytic space, and $\mathcal{F}$ be a coherent analytic sheaf over $X$. Then there is a natural topology on the cohomology groups $H^q(X,\mathcal{F})$ ($0\leq q\leq n$).

In fact, let $\mathcal{U}=\{U_i\}_{i\in I}$ be a Stein covering of $X$. Since $\mathcal{F}$ is coherent, $\mathcal{F}$ is locally isomorphic to a quotient sheaf of some direct sum $\mathcal{O}_X^{\oplus k}$. Since the space of sections $\Gamma(U,\mathcal{O}_X^{\oplus k})$ for any Stein open subset $U\subset X$ can be endowed with the topology of local uniform convergence of holomorphic sections, there is a natural quotient topology on $\Gamma(V,\mathcal{F})$ for any Stein open subset $V\subset X$. Then we consider the product topology on the spaces of \v{C}ech cochains
$C^q(\mathcal{U},\mathcal{F})=\prod\Gamma(U_{i_0i_1\cdots i_q},\mathcal{F})$ and the quotient topology on $\check{H}^q(\mathcal{U},\mathcal{F})$, where $U_{i_0i_1\cdots i_q}$ denotes $U_{i_0}\cap U_{i_1}\cap\cdots\cap U_{i_q}$. Since Leray's Theorem shows that the sheaf (or \v{C}ech) cohomology group $H^q(X,\mathcal{F})$ is isomorphic to $\check{H}^q(\mathcal{U},\mathcal{F})$, $H^q(X,\mathcal{F})$ can be endowed with the resulting topology, which is in fact independent of the choice of the Stein covering $\mathcal{U}$ (cf. Section 4.C of Chapter IX in \cite{Demailly-book}).

The following lemma is a topological result for spaces of sections of a coherent analytic sheaf.

\begin{lem}[Lemma 12 of Section A in Chapter VIII of \cite{Gunning-Rossi}]\label{l:close property of subsheaf}
Let $X$ be a complex analytic space, and $\mathcal{F}$ be a coherent analytic sheaf over $X$. Let $x\in X$, and $\mathcal{M}$ be a submodule of $\mathcal{F}_x$. Then for any open neighborhood $U$ of $x$,
\[M_U:=\{F\in H^0(U,\mathcal{F}):\,\text{the germ of }F\text{ at }x\text{ belongs to } \mathcal{M}\}\]
is a closed subset of $H^0(U,\mathcal{F})$.
\end{lem}

If $X$ is a holomorphically convex complex analytic space, there exists a topological isomorphism between cohomology groups as stated in the following lemma.

\begin{lem}[Lemma II.1 in \cite{Prill1971}]\label{l:Prill1971}
Let $X$ and $S$ be complex analytic spaces, and let $\pi: X\longrightarrow S$ be a proper
holomorphic surjection. Let $\mathcal{F}$ be a coherent analytic sheaf on $X$ and suppose $S$ is a
Stein space. Then the direct image map
\[H^q(X,\mathcal{F})\longrightarrow H^0(S,R^q\pi_*\mathcal{F})\]
is an isomorphism
of topological vector spaces for each $q\geq0$. In particular, each
$H^q(X,\mathcal{F})$ is Hausdorff.
\end{lem}

Let $X$ be an $n$-dimensional complex manifold, $\psi$ be an $L^1_{\mathrm{loc}}$ function on $X$ which is locally bounded above, and $(L,h)$ be a holomorphic line bundle over $X$ equipped with a
singular Hermitian metric $h$. Assume that $\sqrt{-1}\Theta_{L,h}+\sqrt{-1}\partial\bar\partial\psi\geq\gamma$ on $X$ in the sense of currents for some continuous real $(1,1)$-form $\gamma$ on $X$.

Now we define topologies on the $L^2_{\mathrm{loc}}$ Dolbeault cohomology groups with respect to the singular metric $h$.

Let $\mathcal{L}^{n,q}_{(2),h}$ $(0\leq q \leq n)$ denote the sheaf over $X$ whose stalk over a point $x\in X$ consists of germs of Lebesgue measurable $L$-valued $(n,q)$-forms $u$ such that
\[\int_{U}|u|^2_hd\lambda+\int_{U}|\bar\partial u|^2_hd\lambda<+\infty\]
for some open coordinate neighborhood $U$ of $x$, where $d\lambda$ is the Lebesgue measure with respect to the coordinates on $U$. Then
\[0\rightarrow\mathcal{O}_X(K_X\otimes L)\otimes\mathcal{I}(h)\overset{\iota}\rightarrow\mathcal{L}^{n,0}_{(2),h}
\overset{\bar\partial}\rightarrow\mathcal{L}^{n,1}_{(2),h}\overset{\bar\partial}\rightarrow
\mathcal{L}^{n,2}_{(2),h}\overset{\bar\partial}\rightarrow\cdots\]
is a fine resolution of $\mathcal{O}_X(K_X\otimes L)\otimes\mathcal{I}(h)$ by Lemma \ref{l:Hormander-dbar}, where $\iota$ is the inclusion homomorphism.

Let $U\subset X$ be an arbitrary Stein open coordinate subset. We define semi-norms $\|\bullet\|_{h,K}$ on the space of sections $\Gamma(U,\mathcal{L}^{n,q}_{(2),h})$ by
\[\|u\|_{h,K}:=\bigg(\int_{K}|u|^2_hd\lambda+\int_{K}|\bar\partial u|^2_hd\lambda\bigg)^{\frac{1}{2}},\]
where $K$ is any compact subset of $U$. Then $\Gamma(U,\mathcal{L}^{n,q}_{(2),h})$ together with the family of semi-norms $\|\bullet\|_{h,K}$ becomes a Fr\'{e}chet space (the Fr\'{e}chet topology is independent of the choice of the coordinates on $U$). Then the following induced sequence
\begin{equation}\label{e:dbar-sequence of section space-h}
0\rightarrow\Gamma(U,\mathcal{O}_X(K_X\otimes L)\otimes\mathcal{I}(h))\overset{\tilde{\iota}}\rightarrow\Gamma(U,\mathcal{L}^{n,0}_{(2),h})
\overset{\bar\partial_0}\rightarrow\Gamma(U,\mathcal{L}^{n,1}_{(2),h})\overset{\bar\partial_1}\rightarrow\cdots
\end{equation}
is exact by Lemma \ref{l:Hormander-dbar}. The homomorphism $\bar\partial_q$ is continuous for each $q$ by the definitions of the semi-norms $\|\bullet\|_{h,K}$, and the homomorphism $\tilde{\iota}$ is continuous by Lemma \ref{l:local generators uniform estimate}, where $\Gamma(U,\mathcal{O}_X(K_X\otimes L)\otimes\mathcal{I}(h))$ is endowed with the Fr\'{e}chet topology of locally uniform convergence of holomorphic sections.

The corresponding Dolbeault complex is
\begin{equation*}
\Gamma(X,\mathcal{L}^{n,0}_{(2),h})
\overset{\bar\partial_0}\rightarrow\Gamma(X,\mathcal{L}^{n,1}_{(2),h})\overset{\bar\partial_1}\rightarrow
\Gamma(X,\mathcal{L}^{n,2}_{(2),h})\overset{\bar\partial_2}\rightarrow\cdots,
\end{equation*}
where we also endow $\Gamma(X,\mathcal{L}^{n,q}_{(2),h})$ with the Fr\'{e}chet topology defined similarly as the topology of $\Gamma(U,\mathcal{L}^{n,q}_{(2),h})$. Let
\[\bar\partial_{-1}:\,0\rightarrow\Gamma(X,\mathcal{L}^{n,0}_{(2),h})\]
be the zero map. Then the $L^2_{\mathrm{loc}}$ Dolbeault cohomology group with respect to the singular metric $h$ defined by
\[H^{n,q}_{(2)}(X,L,h)=\frac{\mathrm{Ker}\,\bar\partial_q}{\mathrm{Im}\,\bar\partial_{q-1}}\quad(0\leq q\leq n)\]
is endowed with the quotient topology.

For the sheaf (or \v{C}ech) cohomology group and the $L^2_{\mathrm{loc}}$ Dolbeault cohomology group, there is a topological isomorphism as stated in the following lemma.

\begin{lem}\label{l:topological Dolbeault isomorphism-h}
Under the topologies defined above, the group isomorphism
\[H^q(X,\mathcal{O}_X(K_X\otimes L)\otimes\mathcal{I}(h))\simeq H^{n,q}_{(2)}(X,L,h) \quad (0\leq q\leq n)\]
is a topological isomorphism.
\end{lem}

\begin{proof}
The topological isomorphism for $q=0$ follows easily from Lemma \ref{l:local generators uniform estimate}.

Since the homomorphisms in the exact sequence \eqref{e:dbar-sequence of section space-h} are continuous, The topological isomorphism for $q\geq1$ can be proved in the same way as in Proposition 12 in \cite{Andreotti-Kas1973}, where the topological isomorphism between the \v{C}ech cohomology group and the usual $C^\infty$ Dolbeault cohomology group was proved.

\end{proof}

Similarly, we can define topologies on the $L^2_{\mathrm{loc}}$ Dolbeault cohomology groups with respect to the quotient sheaf $\mathcal{I}(h)/\mathcal{I}(he^{-\psi})$.

In fact,
\[0\rightarrow\mathcal{O}_X(K_X\otimes L)\otimes\mathcal{I}(h)/\mathcal{I}(he^{-\psi})\overset{\iota}\rightarrow\mathcal{L}^{n,0}_{(2),h}/\mathcal{L}^{n,0}_{(2),he^{-\psi}}
\overset{\bar\partial}\rightarrow\mathcal{L}^{n,1}_{(2),h}/\mathcal{L}^{n,1}_{(2),he^{-\psi}}\overset{\bar\partial}\rightarrow\cdots\]
is a fine resolution of $\mathcal{O}_X(K_X\otimes L)\otimes\mathcal{I}(h)/\mathcal{I}(he^{-\psi})$ by Lemma \ref{l:Hormander-dbar}, where $\iota$ is the inclusion homomorphism.

Let $U\subset X$ be an arbitrary Stein open coordinate subset. We define semi-norms $\|\bullet\|_{h,K}'$ on the space of sections $\Gamma(U,\mathcal{L}^{n,q}_{(2),h}/\mathcal{L}^{n,q}_{(2),he^{-\psi}})$ by the quotient topology induced from the isomorphism
\[\Gamma(U,\mathcal{L}^{n,q}_{(2),h}/\mathcal{L}^{n,q}_{(2),he^{-\psi}})
\simeq\Gamma(U,\mathcal{L}^{n,q}_{(2),h})/\Gamma(U,\mathcal{L}^{n,q}_{(2),he^{-\psi}}),\]
where $K$ is any compact subset of $U$, i.e.,
\[\|u'\|_{h,K}':=\inf\{\|u\|_{h,K}:\,u\in\Gamma(U,\mathcal{L}^{n,q}_{(2),h})\text{ and }u\text{ is in the equivalent class }u'\}.\]
Then $\Gamma(U,\mathcal{L}^{n,q}_{(2),h}/\mathcal{L}^{n,q}_{(2),he^{-\psi}})$ together with the family of semi-norms $\|\bullet\|_{h,K}'$ becomes a Fr\'{e}chet space. Then the following induced sequence
\begin{eqnarray*}
0&\rightarrow&\Gamma(U,\mathcal{O}_X(K_X\otimes L)\otimes\mathcal{I}(h)/\mathcal{I}(he^{-\psi}))\overset{\tilde{\iota}}\rightarrow\Gamma(U,\mathcal{L}^{n,0}_{(2),h}/\mathcal{L}^{n,0}_{(2),he^{-\psi}})\\
&\overset{\bar\partial_0}\rightarrow&\Gamma(U,\mathcal{L}^{n,1}_{(2),h}/\mathcal{L}^{n,1}_{(2),he^{-\psi}})\overset{\bar\partial_1}\rightarrow\cdots
\end{eqnarray*}
is exact by Lemma \ref{l:Hormander-dbar}. The homomorphism $\bar\partial_q$ is continuous for each $q$ by the definitions of the semi-norms $\|\bullet\|_{h,K}'$, and the homomorphism $\tilde{\iota}$ is continuous by Lemma \ref{l:local generators uniform estimate}, where \[\Gamma(U,\mathcal{O}_X(K_X\otimes L)\otimes\mathcal{I}(h)/\mathcal{I}(he^{-\psi}))\simeq\frac{\Gamma(U,\mathcal{O}_X(K_X\otimes L)\otimes\mathcal{I}(h))}{\Gamma(U,\mathcal{O}_X(K_X\otimes L)\otimes\mathcal{I}(he^{-\psi})}\] is endowed with the quotient topology induced by the Fr\'{e}chet topology of locally uniform convergence of holomorphic sections.

The corresponding Dolbeault complex is
\[\Gamma(X,\mathcal{L}^{n,0}_{(2),h}/\mathcal{L}^{n,0}_{(2),he^{-\psi}})
\overset{\bar\partial_0}\rightarrow\Gamma(X,\mathcal{L}^{n,1}_{(2),h}/\mathcal{L}^{n,1}_{(2),he^{-\psi}})\overset{\bar\partial_1}\rightarrow
\Gamma(X,\mathcal{L}^{n,2}_{(2),h}/\mathcal{L}^{n,2}_{(2),he^{-\psi}})\overset{\bar\partial_2}\rightarrow\cdots,\]
where we also endow $\Gamma(X,\mathcal{L}^{n,q}_{(2),h}/\mathcal{L}^{n,q}_{(2),he^{-\psi}})$ with the Fr\'{e}chet topology defined similarly as the topology of $\Gamma(U,\mathcal{L}^{n,q}_{(2),h}/\mathcal{L}^{n,q}_{(2),he^{-\psi}})$. Let
\[\bar\partial_{-1}:\,0\rightarrow\Gamma(X,\mathcal{L}^{n,0}_{(2),h}/\mathcal{L}^{n,0}_{(2),he^{-\psi}})\]
be the zero map. Then the $L^2_{\mathrm{loc}}$ Dolbeault cohomology group with respect to the quotient sheaf $\mathcal{I}(h)/\mathcal{I}(he^{-\psi})$ defined by
\[H^{n,q}_{(2)}(X,L,h/he^{-\psi})=\frac{\mathrm{Ker}\,\bar\partial_q}{\mathrm{Im}\,\bar\partial_{q-1}}\quad(0\leq q\leq n)\]
is endowed with the quotient topology.

Similarly, we have

\begin{lem}\label{l:topological Dolbeault isomorphism-h/h psi}
Under the topologies defined above, the group isomorphism
\[H^q(X,\mathcal{O}_X(K_X\otimes L)\otimes\mathcal{I}(h)/\mathcal{I}(he^{-\psi}))\simeq H^{n,q}_{(2)}(X,L,h/he^{-\psi}) \quad (0\leq q\leq n)\]
is a topological isomorphism.
\end{lem}

\begin{remark}\label{r:commutative diagram}
It is not hard to obtain the following commutative diagram
\[\begin{CD}
H^q(X,\mathcal{O}_X(K_X\otimes L)\otimes\mathcal{I}(h)) @>p_X>> H^q(X,\mathcal{O}_X(K_X\otimes L)\otimes\mathcal{I}(h)/\mathcal{I}(he^{-\psi}))\\
@Vi_XVV   @Vj_XVV \\
H^{n,q}_{(2)}(X,L,h) @>P_X>> H^{n,q}_{(2)}(X,L,h/he^{-\psi}),
\end{CD}\]
where $i_X$, $j_X$ are the isomorphisms in Lemma \ref{l:topological Dolbeault isomorphism-h} and Lemma \ref{l:topological Dolbeault isomorphism-h/h psi} respectively, and $p_X$, $P_X$ are the natural homomorphisms.\qed
\end{remark}

\bigskip

%%%--------------------------------------------------------

\section{Proof of Theorem \ref{t:Zhou-Zhu-nonreduced-quasipsh}}\label{section-proof-of-theorem-Zhou-Zhu-nonreduced-quasipsh}

In this section, we will denote the sheaf $\mathcal{O}_X(K_X\otimes L)$ simply by $\mathcal{K}$. Then $\mathcal{K}\otimes\mathcal{I}(h)$ and $\mathcal{K}\otimes\mathcal{I}(h)/\mathcal{I}(he^{-\psi})$ are coherent analytic sheaves.

Since $X$ is a holomorphically convex complex manifold, Remmert's reduction theorem implies that there exists a proper holomorphic surjection $\pi:X\longrightarrow S$ such that $S$ is a normal Stein space. Then by Grauert's direct image theorem, the $q$-th direct image sheaf $R^q\pi_*\mathcal{F}$ over $S$ of any coherent analytic sheaf $\mathcal{F}$ over $X$ is coherent and we have the group isomorphism (cf. \cite{Grauert1960})
\[H^q(X,\mathcal{F})\simeq H^0(S,R^q\pi_*\mathcal{F}).\]

Therefore, by the Stein property of $S$ and Cartan's Theorem B, in order to prove the surjectivity statement of Theorem \ref{t:Zhou-Zhu-nonreduced-quasipsh}, it is enough to prove the surjectivity of the sheaf homomorphism
\begin{equation}\label{e:def P}
P:\,R^q\pi_*\big(\mathcal{K}\otimes\mathcal{I}(h)\big)\longrightarrow R^q\pi_*\big(\mathcal{K}\otimes\mathcal{I}(h)/\mathcal{I}(he^{-\psi})\big).
\end{equation}

The proof will be divided into the following five subsections.

Let $y_0\in S$ be an arbitrary fixed point and take
\[f\in R^q\pi_*\big(\mathcal{K}\otimes\mathcal{I}(h)/\mathcal{I}(he^{-\psi})\big)_{y_0}.\]
Then there exists a Stein neighborhood $S_0\subset\subset S$ of $y_0$ such that the germ $f$ has a representation (still denoted by $f$) in $H^0\big(S_0,R^q\pi_*(\mathcal{K}\otimes\mathcal{I}(h)/\mathcal{I}(he^{-\psi}))\big)$, i.e.,
\[f\in H^q\big(X_0,\mathcal{K}\otimes\mathcal{I}(h)/\mathcal{I}(he^{-\psi})\big),\]
where $X_0:=\pi^{-1}(S_0)$.

\subsection{Construction of a smooth
representation $\tilde{f}$ of $f$ from $Y\cap X_0$ to $X_0$, where $Y$ is defined in Remark \ref{r:Y-definition}.}\label{subsection 1} \quad

Let $\mathcal{U}=\{U_i\}_{i\in I}$ be a locally finite Stein covering of $X_0$. By Leray's theorem, $f$ is represented by a \v{C}ech $q$-cocycle $\{c_{i_0\cdots i_q}\}$ with
\[c_{i_0\cdots i_q}\in \Gamma\big(U_{i_0\cdots i_q},\mathcal{K}\otimes\mathcal{I}(h)/\mathcal{I}(he^{-\psi})\big),\]
where $U_{i_0\cdots i_q}$ denotes $U_{i_0}\cap\cdots\cap U_{i_q}$.

Since the natural homomorphism
\[\Gamma\big(U_{i_0\cdots i_q},\mathcal{K}\otimes\mathcal{I}(h)\big)\longrightarrow\Gamma\big(U_{i_0\cdots i_q},\mathcal{K}\otimes\mathcal{I}(h)/\mathcal{I}(he^{-\psi})\big)\]
is surjective by the Stein property of $U_{i_0\cdots i_q}$, we can assume that
\[c_{i_0\cdots i_q}\in\Gamma\big(U_{i_0\cdots i_q},\mathcal{K}\otimes\mathcal{I}(h)\big).\]

By the explicit expression of the isomorphism between \v{C}ech cohomology groups and $C^\infty$ Dolbeault cohomology groups (cf. \cite{Andreotti-Kas1973}, \cite{Demailly-book} or \cite{Cao-Demailly-Matsumura2017}), $f$ is represented by an $(n,q)$-form
\[\tilde{f}:=\sum\limits_{i_0,\cdots, i_q}c_{i_0\cdots i_q}\wedge\xi_{i_q}\bar\partial\xi_{i_0}\wedge\cdots\wedge\bar\partial\xi_{i_{q-1}},\]
where $\{\xi_i\}_{i\in I}$ is a partition of unity subordinate to $\mathcal{U}$.

Let the symbols $\mathcal{L}^{n,q}_{(2),h}$ and $\bar\partial_q$ be as in Section \ref{section-lemmas} (cf. the arguments before Lemma \ref{l:topological Dolbeault isomorphism-h} and Lemma \ref{l:topological Dolbeault isomorphism-h/h psi}). Then $\tilde{f}$ is smooth on $X_0$, $\tilde{f}\in\Gamma(X_0,\mathcal{L}^{n,q}_{(2),h})$ and $\bar\partial_q\tilde{f}=0$ in $\Gamma(X_0,\mathcal{L}^{n,q+1}_{(2),h}/\mathcal{L}^{n,q+1}_{(2),he^{-\psi}})$, i.e.,
\[\bar\partial\tilde{f}\in\Gamma(X_0,\mathcal{L}^{n,q+1}_{(2),he^{-\psi}}).\]

Let $h_0$ be any fixed smooth metric of $L$ on $X$. Then
$h=h_0e^{-\phi}$ for some global function $\phi$ on $X$, which is quasi-psh
by the assumption in the theorem.

Let $S_1\subset\subset S_0$ be any fixed Stein neighborhood of $y_0$ and let $X_1:=\pi^{-1}(S_1)$. Then we have
\[\int_{X_1}\big|\bar\partial \tilde{f}\big|_{\omega,h_0}^2e^{-\phi-\psi}dV_{X,\omega}<+\infty.\]

\subsection{Approximation of singular weights.}\label{subsection 2} \quad

By the assumptions in Theorem \ref{t:Zhou-Zhu-nonreduced-quasipsh}, the following two inequalities hold on $X_0$:
\[\sqrt{-1}\partial\bar\partial\phi+\sqrt{-1}\partial\bar\partial\psi\geq-\sqrt{-1}\Theta_{L,h_0}\]
and
\[\sqrt{-1}\partial\bar\partial\phi+(1+\alpha)\sqrt{-1}\partial\bar\partial\psi\geq-\sqrt{-1}\Theta_{L,h_0},\]
where $\alpha$ can be assumed to be a positive number since $X_0\subset\subset X$.

By the two curvature inequalities above, $\phi+\psi$ and $\phi+(1+\alpha)\psi$ are equal to quasi-psh functions on $X_0$ almost everywhere. Without loss of generality, we can assume that they are quasi-psh on $X_0$.

Since there must exist a continuous
nonnegative $(1,1)$-form $\varpi$ on $X$ such that
\[\bigg(\frac{\sqrt{-1}}{2\pi}\Theta_{T_X}+\varpi\otimes \mathrm{Id}_{T_X}\bigg)
(\kappa_1\otimes \kappa_2,\kappa_1\otimes \kappa_2)\geq0\quad
(\forall\kappa_1,\kappa_2\in T_X)\] holds on a neighborhood of $\overline{X_1}$, by applying Lemma \ref{l:regularization-two-quasipsh}
to the case $\varphi_1:=\phi$, $\varphi_2:=\psi$, $\gamma_1=\gamma_2:=-\sqrt{-1}\Theta_{L,h_0}$ and $\Omega:=X_1$, we obtain two family of upper semicontinuous functions $\{\phi_{\varsigma,\rho}\}$ and $\{\psi_{\varsigma,\rho}\}$ ($\varsigma\in(0,+\infty)$ and $\rho\in(0,\rho_1)$ for some positive number $\rho_1$) defined on a neighborhood of $\overline{X_1}$ satisfying the conclusion of Lemma \ref{l:regularization-two-quasipsh}.

Let $n_1$ be a positive integer such that $n_1\omega\geq\varpi$ on $\overline{X_1}$, and let $\varsigma_\rho:=\frac{\delta_\rho}{n_1\pi}$, where $\delta_\rho$ is as in the conclusion of Lemma \ref{l:regularization-two-quasipsh}. Denote $\phi_{\varsigma_\rho,\rho}$, $\psi_{\varsigma_\rho,\rho}$ and $\Sigma_{\varsigma_\rho}$ (cf. Lemma \ref{l:regularization-two-quasipsh}) simply by $\phi_\rho$, $\psi_\rho$ and $\Sigma_\rho$ respectively. It is obvious that $\Sigma_\rho\subset\Sigma$ for any $\rho\in(0,\rho_1)$ ($\Sigma$ is defined as in Lemma \ref{l:regularization-two-quasipsh}). Then we have
\begin{enumerate}
\item[$(i)$]
$\phi_\rho$ and $\psi_\rho$ are smooth on $\overline{X_1}\setminus\Sigma_\rho$, $\lim\limits_{\rho\rightarrow 0}\phi_\rho=\phi$ on $\overline{X_1}\setminus\Sigma$, and $\lim\limits_{\rho\rightarrow 0}\psi_\rho=\psi$ on $\overline{X_1}\setminus\Sigma$,
\item[$(ii)$]
$\phi_\rho+\psi_\rho$ is quasi-psh on a neighborhood of $\overline{X_1}$ up to its values on $\Sigma_\rho$, increasing with respect to
$\rho$ on $\overline{X_1}\setminus\Sigma$, and converges to $\phi+\psi$ on $\overline{X_1}\setminus\Sigma$ as $\rho\rightarrow 0$,
\item[$(iii)$]
$\phi_\rho\geq\phi$ on $X_1$, $\psi_\rho\leq\sup\limits_{X_0}\psi$ on $X_1$, and $\bar\partial\psi_\rho\in L^1$ on $X_1$,
\item[$(iv)$]
$\sqrt{-1}\Theta_{L,h_0}+\sqrt{-1}\partial\bar\partial\phi_\rho+\sqrt{-1}\partial\bar\partial\psi_\rho
\geq-2\delta_\rho\omega$ on $X_1$,
\item[$(v)$]
$\sqrt{-1}\Theta_{L,h_0}+\sqrt{-1}\partial\bar\partial\phi_\rho+(1+\alpha)\sqrt{-1}\partial\bar\partial\psi_\rho
\geq-2\delta_\rho\omega$ on $X_1$,
\end{enumerate}
where $\{\delta_\rho\}$ is an increasing family of positive numbers
such that $\lim\limits_{\rho\rightarrow 0}\delta_\rho=0$.

\subsection{Construction of additional weights and twist factors.} \quad

For any $t\in(-\infty,0)$, let $\sigma_t$ be the smooth function on $\mathbb{R}$ defined by
\[\sigma_t(s):=\log(e^s+e^t).\]

Without loss of generality, we can assume that $\sup\limits_{X_0}\psi<0$.
Then it follows from the property $(iii)$ in Subsection \ref{subsection 2} that there exists a negative number
$t_0$ such that
$\sigma_t(\psi_\rho)<0$ on $X_1$ for any $t\in(-\infty,t_0)$ and for any $\rho\in(0,\rho_1)$.

Let $\zeta$ and $\chi$ be the solution of the following system of ODEs defined on $(-\infty,0)$:
\begin{numcases}{}
\chi(t)\zeta'(t)-\chi'(t)=1,\label{non-reduced-ode1}\\
\bigg(\chi(t)+\frac{\big(\chi'(t)\big)^2}{\chi(t)
\zeta''(t)-\chi''(t)}\bigg)e^{\zeta(t)+t}=\frac{1}{\alpha}+1,\label{non-reduced-ode2}
\end{numcases}
where we assume that $\zeta$ and $\chi$ are all
smooth on $(-\infty,0)$, and that $\inf\limits_{t<0}\zeta(t)=0$, $\inf\limits_{t<0}\chi(t)=\frac{1}{\alpha}$,
$\zeta'>0$ and $\chi'<0$ on
$(-\infty,0)$. By the similar calculation as in \cite{Guan-Zhou2013b} or \cite{Zhou-Zhu2015}, we can solve the system of ODEs and get the solution
\begin{numcases}{}
\zeta(t)=\log\frac{\frac{1}{\alpha}+1}{\frac{1}{\alpha}+1-e^t},\nonumber\\
\chi(t)=\frac{\frac{1}{\alpha^2}-1+e^t-(\frac{1}{\alpha}+1)t}{\frac{1}{\alpha}+1-e^t}.\nonumber
\end{numcases}

Let $h_{t,\rho}$ be the new metric on the line bundle $L$ over
$X_1\setminus\Sigma_\rho$ defined by
\[h_{t,\rho}:=
h_0e^{-\phi_\rho-\psi_\rho-\zeta(\sigma_t(\psi_\rho))}.\]

Let $\tau_{t,\rho}:=\chi(\sigma_t(\psi_\rho))$ and
$A_{t,\rho}:=\frac{(\chi'(\sigma_t(\psi_\rho)))^2}{\chi(\sigma_t(\psi_\rho))
\zeta''(\sigma_t(\psi_\rho))-\chi''(\sigma_t(\psi_\rho))}$. Set
$\mathrm{B}_{t,\rho} =[\Theta_{t,\rho},\,\Lambda
]$ on $X_1\setminus\Sigma_\rho$, where
\[\Theta_{t,\rho}:=\tau_{t,\rho}\sqrt{-1}
\Theta_{L,h_{t,\rho}}
-\sqrt{-1}\partial\bar\partial\tau_{t,\rho}
-\sqrt{-1}\frac{\partial\tau_{t,\rho}\wedge
\bar\partial\tau_{t,\rho}}{A_{t,\rho}}.\]
We want to prove
\begin{equation}\label{ie:non reduced twisted cuvature}
\Theta_{t,\rho}\big|_{X_1\setminus\Sigma_\rho}\geq
\frac{e^{\psi_\rho+t}}{(e^{\psi_\rho}+e^t)^2}\sqrt{-1}\partial\psi_\rho\wedge
\bar\partial\psi_\rho-2\chi(\sigma_t(\psi_\rho))\delta_\rho\omega.
\end{equation}

It follows from \eqref{non-reduced-ode1} that
\begin{eqnarray*}
& &\Theta_{t,\rho}\big|_{X_1\setminus\Sigma_\rho}\\
&=&\chi(\sigma_t(\psi_\rho))\big(\sqrt{-1}\Theta_{L,h_0}
+\sqrt{-1}\partial\bar\partial
\phi_\rho+\sqrt{-1}\partial\bar\partial\psi_\rho\big)\\
& &+\big(\chi(\sigma_t(\psi_\rho))
\zeta'(\sigma_t(\psi_\rho))
-\chi'(\sigma_t(\psi_\rho))\big)
\sqrt{-1}\partial\bar\partial\sigma_t(\psi_\rho)\\
&=&\chi(\sigma_t(\psi_\rho))\big(\sqrt{-1}\Theta_{L,h_0}
+\sqrt{-1}\partial\bar\partial
\phi_\rho+\sqrt{-1}\partial\bar\partial\psi_\rho\big)+\sqrt{-1}\partial\bar\partial
(\sigma_t(\psi_\rho))\\
&=&\chi(\sigma_t(\psi_\rho))\big(\sqrt{-1}\Theta_{L,h_0}
+\sqrt{-1}\partial\bar\partial
\phi_\rho+\sqrt{-1}\partial\bar\partial\psi_\rho\big)\\
& &+\frac{e^{\psi_\rho}}{e^{\psi_\rho}+e^t}\sqrt{-1}\partial\bar\partial\psi_\rho
+\frac{e^{\psi_\rho+t}}{(e^{\psi_\rho}+e^t)^2}\sqrt{-1}\partial\psi_\rho\wedge
\bar\partial\psi_\rho.
\end{eqnarray*}
Since $\chi\geq\chi(0)=\frac{1}{\alpha}$ on $(-\infty,0)$, it follows from the properties $(iv)$ and $(v)$ in Subsection \ref{subsection 2} that
\begin{eqnarray*}
& &\chi(\sigma_t(\psi_\rho))\big(\sqrt{-1}\Theta_{L,h_0}
+\sqrt{-1}\partial\bar\partial
\phi_\rho+\sqrt{-1}\partial\bar\partial\psi_\rho\big)
+\frac{e^{\psi_\rho}}{e^{\psi_\rho}+e^t}\sqrt{-1}\partial\bar\partial\psi_\rho\\
&=&\chi(\sigma_t(\psi_\rho))\big(\sqrt{-1}\Theta_{L,h_0}
+\sqrt{-1}\partial\bar\partial
\phi_\rho+\sqrt{-1}\partial\bar\partial\psi_\rho+2\delta_\rho\omega\big)
-2\chi(\sigma_t(\psi_\rho))\delta_\rho\omega\\
& &+\frac{e^{\psi_\rho}}{\alpha(e^{\psi_\rho}+e^t)}\cdot\alpha\sqrt{-1}\partial\bar\partial\psi_\rho\\
&\geq&\frac{e^{\psi_\rho}}{\alpha(e^{\psi_\rho}+e^t)}
\big(\sqrt{-1}\Theta_{L,h_0}
+\sqrt{-1}\partial\bar\partial
\phi_\rho+\sqrt{-1}\partial\bar\partial\psi_\rho+2\delta_\rho\omega+\alpha\sqrt{-1}\partial\bar\partial\psi_\rho\big)\\
& &-2\chi(\sigma_t(\psi_\rho))\delta_\rho\omega\\
&\geq&-2\chi(\sigma_t(\psi_\rho))\delta_\rho\omega
\end{eqnarray*}
on $X_1\setminus\Sigma_\rho$. Hence we get \eqref{ie:non reduced twisted cuvature} as desired.

We choose an increasing family of
positive numbers
$\{\rho_t\}_{t\in(-\infty,t_0)}$ such that $\rho_t<\rho_1$ for any $t$,
$\lim\limits_{t\rightarrow-\infty}\rho_t=0$, and
\begin{equation}\label{ie:non reduced delta1}
2\chi(t)\delta_{\rho_t}<\frac{e^{2t}}{q+1}\text{ for any }t.
\end{equation}

Since $\sigma_t(\psi_\rho)\geq t$ on
$X_1$ and $\chi$ is decreasing, we have
$\chi(\sigma_t(\psi_\rho))\leq \chi(t)$
on $X_1$. Then it follows from \eqref{ie:non reduced twisted cuvature} and
\eqref{ie:non reduced delta1} that
\[\Theta_{t,\rho}
\big|_{X_1\setminus\Sigma_\rho}\geq
\frac{e^{\psi_\rho+t}}{(e^{\psi_\rho}+e^t)^2}\sqrt{-1}\partial\psi_\rho\wedge
\bar\partial\psi_\rho-\frac{e^{2t}}{q+1}\omega,\quad\forall\,\rho\in(0,\rho_t].\]
Hence
\begin{equation}\label{ie:non reduced B-estimate}
\mathrm{B}_{t,\rho}
+e^{2t}\mathrm{I}
\geq\bigg[\frac{e^{\psi_\rho+t}}{(e^{\psi_\rho}+e^t)^2}\sqrt{-1}\partial\psi_\rho\wedge
\bar\partial\psi_\rho,\,\Lambda \bigg]=\frac{e^{\psi_\rho+t}}{(e^{\psi_\rho}+e^t)^2}\mathrm{T}_{\bar\partial\psi_\rho}
\mathrm{T}_{\bar\partial\psi_\rho}^*\geq0
\end{equation}
holds on $X_1\setminus\Sigma_\rho$ for any $\rho\in(0,\rho_t]$ as an operator on
$(n,q+1)$-forms, where $\mathrm{T}_{\bar\partial\psi_\rho}$ denotes the
operator $\bar\partial\psi_\rho\wedge\bullet$ and
$\mathrm{T}_{\bar\partial\psi_\rho}^*$ is its Hilbert adjoint
operator.

\subsection{Construction of suitably truncated forms and solving
$\bar\partial$ globally with $L^2$ estimates.}\label{subsection 4} \quad

It is easy to construct a smooth function
$\theta:\mathbb{R}\longrightarrow[0,1]$ such that $\theta=1$ on
$(-\infty,0]$, $\theta=0$ on $[1,+\infty)$
and $|\theta'|\leq 2$ on $\mathbb{R}$.

Define $g_{t,\rho}=\bar\partial
\big(\theta(\psi_\rho-t)
\tilde{f}\big)$ on $X_1\setminus\Sigma_\rho$, where $\tilde{f}$ is as in Subsection \ref{subsection 1}. Then $\bar\partial g_{t,\rho}=0$ on $X_1\setminus\Sigma_\rho$ and
\[g_{t,\rho}=g_{1,t,\rho}+g_{2,t,\rho},\]
where $g_{1,t,\rho}:=\theta'(\psi_\rho-t)\bar\partial\psi_\rho\wedge\tilde{f}$ and
$g_{2,t,\rho}:=\theta(\psi_\rho-t)\bar\partial \tilde{f}$.

The Cauchy-Schwarz inequality and \eqref{ie:non reduced B-estimate} imply that, on $X_1\setminus\Sigma_\rho$,
\begin{eqnarray}
& &\langle(\mathrm{B}_{t,\rho}+2e^{2t} \mathrm{I})^{-1}
g_{t,\rho},g_{t,\rho}
\rangle_{\omega,h_{t,\rho}}\label{ie:non reduced estimate with delta}\\
&\leq&2\langle(\mathrm{B}_{t,\rho}+2e^{2t}
\mathrm{I})^{-1} g_{1,t,\rho},g_{1,t,\rho}
\rangle_{\omega,h_{t,\rho}} +2\langle(\mathrm{B}_{t,\rho}+2e^{2t}
\mathrm{I})^{-1} g_{2,t,\rho},g_{2,t,\rho}
\rangle_{\omega,h_{t,\rho}}\nonumber\\
&\leq&2\langle(\mathrm{B}_{t,\rho}+e^{2t}
\mathrm{I})^{-1} g_{1,t,\rho},g_{1,t,\rho}
\rangle_{\omega,h_{t,\rho}} +2\langle\frac{1}{e^{2t}}
g_{2,t,\rho},g_{2,t,\rho}
\rangle_{\omega,h_{t,\rho}}\nonumber
\end{eqnarray}
for any $\rho\in(0,\rho_t]$.

Since $|\theta'|\leq 2$ on $\mathbb{R}$, $\theta'=0$ on $\mathbb{R}\setminus(0,1)$, $\zeta\geq0$ and $\phi_\rho\geq\phi$ on $X_1$ (cf. the property $(iii)$ in Subsection \ref{subsection 2}), we obtain from \eqref{ie:non reduced B-estimate} that
\begin{eqnarray*}
& &\int_{X_1\setminus\Sigma_\rho}\langle
(\mathrm{B}_{t,\rho}+e^{2t} \mathrm{I})^{-1}
g_{1,t,\rho},g_{1,t,\rho}
\rangle_{\omega,h_{t,\rho}}dV_{X,\omega}\\
&\leq&\int_{X_1}
\frac{4(e^{\psi_\rho}+e^t)^2}
{e^{2\psi_\rho+t}}|\tilde{f}|^2_{\omega,h_0}
e^{-\phi_\rho}\mathbb{I}_{\{t<\psi_\rho<t+1\}}dV_{X,\omega}\\
&\leq&\frac{4(e+1)^2}{e^t}\int_{X_1}
|\tilde{f}|^2_{\omega,h_0}
e^{-\phi}\mathbb{I}_{\{t<\psi_\rho<t+1\}}dV_{X,\omega}
\end{eqnarray*}
for any $\rho\in(0,\rho_t]$, where $\mathbb{I}_{\{t<\psi_\rho<t+1\}}$ is the characteristic function associated to the set $\{t<\psi_\rho<t+1\}$.

Since $\theta=0$ on $[1,+\infty)$, $\zeta\geq0$ and
$\phi_\rho+\psi_\rho\geq\phi+\psi$ on
$X_1\setminus\Sigma$ (cf. the property $(ii)$ in Subsection \ref{subsection 2}), we get
\[\int_{X_1\setminus\Sigma_\rho}\langle\frac{1}
{e^{2t}}g_{2,t,\rho},g_{2,t,\rho}
\rangle_{\omega,h_{t,\rho}}dV_{X,\omega}
\leq\frac{1}{e^{2t}}\int_{X_1}
|\bar\partial \tilde{f}|^2_{\omega,h_0}e^{-\phi-\psi}\mathbb{I}_{\{\psi_\rho<t+1\}}dV_{X,\omega}.\]

Therefore, it follows from \eqref{ie:non reduced estimate with delta} that
\begin{eqnarray*}
C_\rho(t)&:=&\int_{X_1\setminus\Sigma_\rho}\langle(\mathrm{B}_{t,\rho}
+2e^{2t} \mathrm{I})^{-1}
g_{t,\rho},g_{t,\rho}
\rangle_{\omega,h_{t,\rho}}dV_{X,\omega}\\
&\leq&\frac{8(e+1)^2}{e^t}\int_{X_1}
|\tilde{f}|^2_{\omega,h_0}
e^{-\phi}\mathbb{I}_{\{t<\psi_\rho<t+1\}}dV_{X,\omega}\\
& &+\frac{2}{e^{2t}}\int_{X_1}
|\bar\partial \tilde{f}|^2_{\omega,h_0}e^{-\phi-\psi}\mathbb{I}_{\{\psi_\rho<t+1\}}dV_{X,\omega}
\end{eqnarray*}
for any $\rho\in(0,\rho_t]$.

Since we have obtained
\[\int_{X_1}|\tilde{f}|^2_{\omega,h_0}e^{-\phi}dV_{X,\omega}<+\infty\]
and
\[\int_{X_1}|\bar\partial \tilde{f}|^2_{\omega,h_0}e^{-\phi-\psi}dV_{X,\omega}<+\infty\]
in Subsection \ref{subsection 1}, it follows from the property $\lim\limits_{\rho\rightarrow0}\psi_\rho=\psi$ on $X_1\setminus\Sigma$ (cf. property $(i)$ in Subsection \ref{subsection 2}) and Fatou's lemma that
\begin{eqnarray}
C(t):=\varlimsup\limits_{\rho\rightarrow0}C_\rho(t)
&\leq&\frac{8(e+1)^2}{e^t}\int_{X_1}
|\tilde{f}|^2_{\omega,h_0}
e^{-\phi}\mathbb{I}_{\{t\leq\psi\leq t+1\}}dV_{X,\omega}\label{ie:non reduced estimate-C-t}\\
& &+\frac{2}{e^{2t}}\int_{X_1}
|\bar\partial \tilde{f}|^2_{\omega,h_0}e^{-\phi-\psi}\mathbb{I}_{\{\psi\leq t+1\}}dV_{X,\omega}\nonumber
\end{eqnarray}
and
\begin{equation}\label{e:non reduced limit-C-t}
\lim\limits_{t\rightarrow-\infty}e^{2t}C(t)=0.
\end{equation}

Since $X_1$ is a holomorphically convex K\"{a}hler manifold, $X_1$ carries a complete K\"{a}hler metric. Hence $X_1\setminus\Sigma_\rho$ carries a complete K\"{a}hler metric by Lemma \ref{l:Demailly1982-complete-metric}. Then by Lemma \ref{l:Demailly-non complete metric}, there exists an $L$-valued $(n,q)$-form $u_{t,\rho}$ and an $L$-valued $(n,q+1)$-form $v_{t,\rho}$ such that
\begin{equation}\label{e:non reduced dbar}
\bar\partial u_{t,\rho}
+\sqrt{2}e^tv_{t,\rho} =g_{t,\rho}\quad\text{on}\quad X_1\setminus\Sigma_\rho
\end{equation}
and
\begin{eqnarray*}
& &\int_{X_1\setminus\Sigma_\rho}\frac{|u_{t,\rho}|^2
_{\omega,h_0}e^{-\phi_\rho-\psi_\rho-\zeta(\sigma_t(\psi_\rho))}}
{\tau_{t,\rho}+A_{t,\rho}}dV_{X,\omega}\\
& &+\int_{X_1\setminus\Sigma_\rho}|v_{t,\rho}|^2_{\omega,h_0}
e^{-\phi_\rho-
\psi_\rho-\zeta(\sigma_t(\psi_\rho))}dV_{X,\omega}\\
&\leq& C_\rho(t)
\end{eqnarray*}
for any $\rho\in(0,\rho_t]$.

Since
\[\frac{e^{-\zeta(\sigma_t(\psi_\rho))}}{\tau_{t,\rho}+A_{t,\rho}}=\frac{\alpha e^{\sigma_t(\psi_\rho)}}{1+\alpha}\geq\frac{\alpha e^t}{1+\alpha}\]
by \eqref{non-reduced-ode2} and
\[e^{-\zeta(\sigma_t(\psi_\rho))}\geq e^{-\zeta(0)}=\frac{1}{1+\alpha},\]
we get
\begin{equation}\label{ie:non reduced estimate-u-and-h}
\int_{X_1}\frac{\alpha e^t|u_{t,\rho}|^2
_{\omega,h_0}e^{-\phi_\rho-\psi_\rho}}
{1+\alpha}dV_{X,\omega}+\int_{X_1}\frac{|v_{t,\rho}|^2_{\omega,h_0}
e^{-\phi_\rho-
\psi_\rho}}{1+\alpha}dV_{X,\omega}
\leq C_\rho(t)
\end{equation}
for any $\rho\in(0,\rho_t]$.

Since $\sup\limits_{\overline{X_1}\setminus\Sigma}(\phi_\rho+\psi_\rho)<+\infty$ (cf. the property $(ii)$ in Subsection \ref{subsection 2}), we get that
$u_{t,\rho}\in
L^2$ and $v_{t,\rho}\in
L^2$. Since $\bar\partial\psi_\rho\in L^1$ on $X_1$ (cf. the property $(iii)$ in Subsection \ref{subsection 2}), we have $g_{t,\rho}=g_{1,t,\rho}+g_{2,t,\rho}\in L^1$ on $X_1$. Then it follows from
\eqref{e:non reduced dbar} and Lemma \ref{l:extension} that
\begin{equation}\label{e:non reduced dbar on X_1}
\bar\partial u_{t,\rho}
+\sqrt{2}e^tv_{t,\rho} =g_{t,\rho}=\bar\partial
\big(\theta(\psi_\rho-t)
\tilde{f}\big)\quad\text{on}\quad X_1.
\end{equation}

Since $\lim\limits_{\rho\rightarrow0}\psi_\rho=\psi$ on $X_1\setminus\Sigma$ (cf. the property $(i)$ in Subsection \ref{subsection 2}), $\theta(\psi_\rho-t)
\tilde{f}$ converges to $\theta(\psi-t)
\tilde{f}$ in $L^2$ as $\rho\rightarrow0$ by Lebesgue's dominated convergence theorem. Hence $\bar\partial
\big(\theta(\psi_\rho-t)
\tilde{f}\big)$ converges to $\bar\partial
\big(\theta(\psi-t)
\tilde{f}\big)$ in the sense of currents as $\rho\rightarrow0$.

Let $t$ be fixed and let $\{\rho_j\}_{j=2}^{+\infty}$ be a decreasing sequence of positive numbers such that $\rho_2<\rho_t$ and $\lim\limits_{j\rightarrow+\infty}\rho_j=0$.

Since $\phi_{\rho_j}+\psi_{\rho_j}$ decreases to $\phi+\psi$ on $\overline{X_1}\setminus\Sigma$ as $j\rightarrow+\infty$ (cf. the property $(ii)$ in Subsection \ref{subsection 2}), by extracting weak limits of $\{u_{t,\rho_j}\}_{j=2}^{+\infty}$ as $j\rightarrow+\infty$, it follows from \eqref{ie:non reduced estimate-u-and-h} and the diagonal argument that there exists an $L^2$ $L$-valued $(n,q)$-form $u_t$ such that a subsequence $\{u_{t,\rho_{j_r}}\}_{r=1}^{+\infty}$ of $\{u_{t,\rho_j}\}_{j=2}^{+\infty}$ converges to $u_t$ weakly in $L^2_{(n,q)}(X_1,e^{-\phi_{\rho_i}-\psi_{\rho_i}}dV_{X,\omega})$ for any positive integer $i$ as $r\rightarrow+\infty$.

Hence the Banach-Steinhaus Theorem implies that, for any positive integer $i$,
\begin{eqnarray*}
\int_{X_1}\frac{\alpha e^t|u_t|^2_{\omega,h_0}e^{-\phi_{\rho_i}-\psi_{\rho_i}}}{1+\alpha}dV_{X,\omega}&\leq&
\varliminf_{r\rightarrow+\infty}\int_{X_1}\frac{\alpha e^t|u_{t,\rho_{j_r}}|^2_{\omega,h_0}e^{-\phi_{\rho_i}-\psi_{\rho_i}}}{1+\alpha}dV_{X,\omega}\\
&\leq&\varliminf_{r\rightarrow+\infty}\int_{X_1}\frac{\alpha e^t|u_{t,\rho_{j_r}}|^2_{\omega,h_0}e^{-\phi_{\rho_{j_r}}-\psi_{\rho_{j_r}}}}{1+\alpha}dV_{X,\omega}.
\end{eqnarray*}
Then Fatou's lemma implies that
\begin{eqnarray*}
\int_{X_1}\frac{\alpha e^t|u_t|^2_{\omega,h_0}e^{-\phi-\psi}}{1+\alpha}dV_{X,\omega}&\leq&
\varliminf_{i\rightarrow+\infty}\int_{X_1}\frac{\alpha e^t|u_t|^2_{\omega,h_0}e^{-\phi_{\rho_i}-\psi_{\rho_i}}}{1+\alpha}dV_{X,\omega}\\
&\leq&\varliminf_{r\rightarrow+\infty}\int_{X_1}\frac{\alpha e^t|u_{t,\rho_{j_r}}|^2_{\omega,h_0}e^{-\phi_{\rho_{j_r}}-\psi_{\rho_{j_r}}}}{1+\alpha}dV_{X,\omega}.
\end{eqnarray*}

Similar weak limit argument can also be applied to subsequences of $\{v_{t,\rho_j}\}_{j=2}^{+\infty}$.

In conclusion, by \eqref{e:non reduced dbar on X_1}, \eqref{ie:non reduced estimate-u-and-h} and \eqref{ie:non reduced estimate-C-t}, there exist $L^2$ forms $u_t$ and $v_t$ such that
\begin{equation}\label{e:non reduced dbar on X_1 limit-1}
\bar\partial u_t
+\sqrt{2}e^tv_t=\bar\partial
\big(\theta(\psi-t)
\tilde{f}\big)\quad\text{on}\quad X_1
\end{equation}
and
\begin{equation}\label{ie:non reduced estimate-u-and-h no rho}
\int_{X_1}\frac{\alpha e^t|u_t|^2
_{\omega,h_0}e^{-\phi-\psi}}
{1+\alpha}dV_{X,\omega}+\int_{X_1}\frac{|v_t|^2_{\omega,h_0}
e^{-\phi-
\psi}}{1+\alpha}dV_{X,\omega}
\leq C(t).
\end{equation}

\eqref{e:non reduced limit-C-t} and \eqref{ie:non reduced estimate-u-and-h no rho} imply that
\begin{equation}\label{e:non reduced error term limit-1}
\lim\limits_{t\rightarrow-\infty}\int_{X_1}|\sqrt{2}e^tv_t|^2_{\omega,h_0}e^{-\phi-\psi}dV_{X,\omega}
\leq\lim\limits_{t\rightarrow-\infty}2(1+\alpha)e^{2t}C(t)=0.
\end{equation}

\subsection{Final conclusion.} \quad

Some ideas in this subsection come from \cite{Cao-Demailly-Matsumura2017}.

By Lemma \ref{l:topological Dolbeault isomorphism-h}, Lemma \ref{l:topological Dolbeault isomorphism-h/h psi} and the commutative diagram in Remark \ref{r:commutative diagram}, we will identify the sheaf (or \v{C}ech) cohomology groups with the $L^2_{\mathrm{loc}}$ Dolbeault cohomology groups, i.e., $i_{X_1}=\mathrm{Id}$, $j_{X_1}=\mathrm{Id}$ and $p_{X_1}=P_{X_1}$.

Since $H^{q+1}(X_1,\mathcal{K}\otimes\mathcal{I}(h))$ ($0\leq q\leq n$) is Hausdorff by Lemma \ref{l:Prill1971},
Lemma \ref{l:topological Dolbeault isomorphism-h} implies that the image of the map
\[\bar\partial_q:\,\Gamma(X_1,\mathcal{L}^{n,q}_{(2),h})\longrightarrow\Gamma(X_1,\mathcal{L}^{n,q+1}_{(2),h})\quad(0\leq q\leq n)\]
is closed. Hence
\[\Gamma(X_1,\mathcal{L}^{n,q}_{(2),h})\overset{\bar\partial_{q}}\longrightarrow\mathrm{Im}\,\bar\partial_q\quad(0\leq q\leq n)\]
is a continuous linear surjection between Fr\'{e}chet spaces. Therefore, the open mapping theorem (Lemma \ref{l:open mapping theorem}) implies that this map is open. Hence $\bar\partial_q$ induces a topological isomorphism of Fr\'{e}chet spaces
\begin{equation}\label{e:isomorphism induced by dbar}
\frac{\Gamma(X_1,\mathcal{L}^{n,q}_{(2),h})}{\mathrm{Ker}\,\bar\partial_q}\simeq\mathrm{Im}\,\bar\partial_q\quad(0\leq q\leq n).
\end{equation}

Let $w_t:=\sqrt{2}e^tv_t$. Then \eqref{e:non reduced dbar on X_1 limit-1} and \eqref{ie:non reduced estimate-u-and-h no rho} implies that
\[w_t=\bar\partial
\big(\theta(\psi-t)
\tilde{f}-u_t\big)\in \mathrm{Im}\,\bar\partial_q.\]

Since \eqref{e:non reduced error term limit-1} implies that
\[\lim\limits_{t\rightarrow-\infty}\int_{X_1}|w_t|^2_{\omega,h}dV_{X,\omega}=0,\]
it follows from \eqref{e:isomorphism induced by dbar} that there exists a sequence of $L$-valued $(n,q)$-forms $s_t\in\Gamma(X_1,\mathcal{L}^{n,q}_{(2),h})$ such that
$\bar\partial s_t=w_t$ on $X_1$ and
\[\lim\limits_{t\rightarrow-\infty}\|s_t\|_{h,K}=0\]
for any compact set $K$ in open coordinate charts of $X_1$, where $\|\bullet\|_{h,K}$ is defined just before Lemma \ref{l:topological Dolbeault isomorphism-h}.

Let $\tilde{f}_t:=\theta(\psi-t)
\tilde{f}-u_t-s_t$. Then $\bar\partial\tilde{f}_t=0$. Hence
\[\tilde{f}_t\in\mathrm{Ker}\,\bar\partial_q\subset \Gamma(X_1,\mathcal{L}^{n,q}_{(2),h}).\]

As explained just before Lemma \ref{l:topological Dolbeault isomorphism-h/h psi}, the Fr\'{e}chet topology on \[\Gamma(X_1,\mathcal{L}^{n,q}_{(2),h}/\mathcal{L}^{n,q}_{(2),he^{-\psi}})\]
is defined by the semi-norms $\|\bullet\|_{h,K}'$ for any compact set $K$ in open coordinate charts of $X_1$, and the topology of the $q$-th $L^2_{\mathrm{loc}}$ Dolbeault cohomology group
\[H^{n,q}_{(2)}(X_1,L,h/he^{-\psi})\]
is obtained as the quotient topology induced from the semi-norms $\|\bullet\|_{h,K}'$.

Let $\|\bullet\|_{h,K}''$ denote the induced semi-norms on
$H^{n,q}_{(2)}(X_1,L,h/he^{-\psi})$.
Then $\|\bullet\|_{h,K}''$ is smaller than $\|\bullet\|_{h,K}'$ in some sense.

Since \eqref{ie:non reduced estimate-u-and-h no rho} and the construction in Subsection \ref{subsection 1} imply that
\[\theta(\psi-t)
\tilde{f}-u_t-\tilde{f}\in L^2(X_1,\wedge^{n,q}T_X^*\otimes L,he^{-\psi})\]
and
\[\bar\partial(\theta(\psi-t)
\tilde{f}-u_t-\tilde{f})=w_t-\bar\partial\tilde{f}\in L^2(X_1,\wedge^{n,q+1}T_X^*\otimes L,he^{-\psi}),\]
we get
\[\theta(\psi-t)
\tilde{f}-u_t-\tilde{f}\in\Gamma(X_1,\mathcal{L}^{n,q}_{(2),he^{-\psi}}).\]
Then
\[\tilde{f}_t-\tilde{f}=-s_t\quad\mathrm{mod}\quad\Gamma(X_1,\mathcal{L}^{n,q}_{(2),he^{-\psi}}).\]
Hence
\[\|\tilde{f}_t-\tilde{f}\|_{h,K}''\leq \|\tilde{f}_t-\tilde{f}\|_{h,K}'\leq\|s_t\|_{h,K}\rightarrow0\text{ as }t\rightarrow-\infty,\]
where $K$ is any compact subset in open coordinate charts of $X_1$.

Hence Lemma \ref{l:Prill1971} implies that $f$ belongs to the closure of $H^0(S_1,\mathrm{Im}\,P)$ in $H^0\big(S_1,R^q\pi_*(\mathcal{K}\otimes\mathcal{I}(h)/\mathcal{I}(he^{-\psi}))\big)$, where $P$ is the sheaf homomorphism in \eqref{e:def P} and $\mathrm{Im}\,P$ is the image of $P$.

Since
\begin{align*}
H^0(S_1,\mathrm{Im}\,P)=\underset{y\in S_1}\bigcap\big\{F\in\,& H^0\big(S_1,R^q\pi_*(\mathcal{K}\otimes\mathcal{I}(h)/\mathcal{I}(he^{-\psi}))\big);\\
&\text{the germ of }F\text{ at }y\text{ belongs to }(\mathrm{Im}\,P)_y\big\},
\end{align*}
Lemma \ref{l:close property of subsheaf} implies that $H^0(S_1,\mathrm{Im}\,P)$ is a closed subset of \[H^0\big(S_1,R^q\pi_*(\mathcal{K}\otimes\mathcal{I}(h)/\mathcal{I}(he^{-\psi}))\big).\]
Hence $f\in H^0(S_1,\mathrm{Im}\,P)$.

Since $f$ is arbitrary, $P$ is surjective. Thus Theorem \ref{t:Zhou-Zhu-nonreduced-quasipsh} is proved.

\bigskip

%%%--------------------------------------------------------

\section{Proof of Theorem \ref{t:Zhou-Zhu-optimal-jet-extension}}\label{section-proof-of-theorem-optimal-jet-extension}

We will only give the proof when $\alpha\in(0,+\infty)$. The proof for the case $\alpha=+\infty$ is almost the same.

The proof will be divided into five subsections. Some arguments in the proof will be similar as those in Section \ref{section-proof-of-theorem-Zhou-Zhu-nonreduced-quasipsh}.

\subsection{Construction of a smooth
extension $\tilde{f}$ of $f$ from $Y\cap X$ to $X$.}\label{jet-subsection 1} \quad

Since $f\in H^0\big(X,\mathcal{O}_X(K_X\otimes L)\otimes\mathcal{I}'_{\psi}(h)/\mathcal{I}(he^{-\psi})\big)$, there exists a locally finite covering $\mathcal{U}=\{U_i\}_{i\in I}$ of $X$ by coordinate balls, and a family of holomorphic sections
\[f_i\in\Gamma\big(U_i,\mathcal{O}_X(K_X\otimes L)\otimes\mathcal{I}'_{\psi}(h)\big)\quad(i\in I)\]
such that $f$ is the images of $\{f_i\}_{i\in I}$ under the natural morphisms
\[\Gamma\big(U_i,\mathcal{O}_X(K_X\otimes L)\otimes\mathcal{I}'_{\psi}(h)\big)\longrightarrow\Gamma\big(U_i,\mathcal{O}_X(K_X\otimes L)\otimes\mathcal{I}'_{\psi}(h)/\mathcal{I}(he^{-\psi})\big),\quad i\in I.\]
Hence
\[f_i-f_j\in\Gamma\big(U_i\cap U_j,\mathcal{O}_X(K_X\otimes L)\otimes\mathcal{I}(he^{-\psi})\big),\quad\forall i,j\in I.\]

Let $\{\xi_i\}_{i\in I}$ is a partition of unity subordinate to $\mathcal{U}$, and let
\[\tilde{f}:=\sum\limits_{i\in I}\xi_if_i.\]
Then $\tilde{f}$ is smooth on $X$, and
\[\bar\partial\tilde{f}|_{U_j}=\bar\partial\tilde{f}-\bar\partial f_j=\bar\partial(\sum\limits_{i\in I}\xi_if_i)-\bar\partial(\sum\limits_{i\in I}\xi_i f_j)=\sum\limits_{i\in I}\bar\partial\xi_i\wedge(f_i-f_j),\quad\forall j\in I.\]

Since $X$ is weakly pseudoconvex, there exists a smooth psh exhaustion function $\Psi$ on $X$. Let $X_k:=\{x\in X:\,\Psi(x)<k\}$ ($k\in \mathbb{Z}^+$ and $k\geq3$, we choose $\Psi$ such that $X_3\neq\emptyset$).

Let $h_0$ be any fixed smooth metric of $L$ on $X$. Then
$h=h_0e^{-\phi}$ for some global function $\phi$ on $X$, which is quasi-psh
by the assumption in the theorem. Then we have
\begin{equation}\label{ie:jet f int finite}
\int_{X_k}|\tilde{f}|^2_{\omega,h_0}e^{-\phi}dV_{X,\omega}<+\infty
\end{equation}
and
\begin{equation}\label{ie:jet dbar f int finite}
\int_{X_k}\big|\bar\partial\tilde{f}\big|_{\omega,h_0}^2e^{-\phi-\psi}dV_{X,\omega}<+\infty.
\end{equation}

In the following subsections, $k$ will be fixed until the end of the proof ($k\in \mathbb{Z}^+$ and $k\geq3$).

\subsection{Approximation of singular weights.}\label{jet-subsection 2} \quad

This subsection is almost the same as Subsection \ref{subsection 2}. We need only to replace $X_0$ and $X_1$ in Subsection \ref{subsection 2} by $X_{k+1}$ and $X_k$ respectively, and we obtain two family of upper semicontinuous functions $\{\phi_\rho\}_{\rho\in(0,\rho_1)}$ and $\{\psi_\rho\}_{\rho\in(0,\rho_1)}$ ($\rho_1$ is some positive number) defined on a neighborhood of $\overline{X_k}$ satisfying the same properties $(i)-(v)$ as in Subsection \ref{subsection 2}.

\subsection{Construction of additional weights and twist factors.} \quad

Let $\varrho:\mathbb{R}\rightarrow[0,+\infty)$ be
the function defined by
\begin{displaymath}
\varrho(t)=\begin{cases}
\big(\int_{-1}^1
e^{\frac{1}{t^2-1}}dt\big)^{-1}e^{\frac{1}{t^2-1}} &\quad \mathrm{if}\quad|t|<1\\
0 &\quad \mathrm{if}\quad |t|\geq1.
\end{cases}
\end{displaymath}
Let $\varepsilon\in(0,\frac{1}{4})$ and let $\varrho_\varepsilon:\mathbb{R}\rightarrow[0,+\infty)$ be the function defined by
$\varrho_\varepsilon(t)=\frac{4}{\varepsilon}\varrho(\frac{4}{\varepsilon}t)$.
Then $\varrho_\varepsilon$ is smooth on $\mathbb{R}$ with
support contained in $[-\frac{1}{4}\varepsilon,\frac{1}{4}\varepsilon]$ and
$\int_{-\infty}^{+\infty}\varrho_\varepsilon(t)dt=1$.

Let $\{\sigma_{\varepsilon,t}\}_{\varepsilon\in(0,\frac{1}{4}),t\in(-\infty,-1)}$ be the family of functions on $\mathbb{R}$ defined by (see Section 5 in \cite{Guan-Zhou2013b})
\begin{eqnarray*}
\sigma_{\varepsilon,t}(s)&=&\int_{-\infty}^s\bigg(\int_{-\infty}^{t_2}\frac{1}{1-2\varepsilon}(\mathbb{I}_{(t+\varepsilon,t+1-\varepsilon)}
\ast\varrho_\varepsilon)(t_1)dt_1\bigg)dt_2\\
& &-\int_{-\infty}^0\bigg(\int_{-\infty}^{t_2}\frac{1}{1-2\varepsilon}(\mathbb{I}_{(t+\varepsilon,t+1-\varepsilon)}
\ast\varrho_\varepsilon)(t_1)dt_1\bigg)dt_2,
\end{eqnarray*}
where $\mathbb{I}_{(t+\varepsilon,t+1-\varepsilon)}$ is the characteristic function associated to the interval $(t+\varepsilon,t+1-\varepsilon)$, and the notation $\ast$ denotes the convolution of two functions.

Then we have
\[\sigma_{\varepsilon,t}'(s)=\int_{-\infty}^s\frac{1}{1-2\varepsilon}(\mathbb{I}_{(t+\varepsilon,t+1-\varepsilon)}\ast\varrho_\varepsilon)(t_1)dt_1\]
and
\[\sigma_{\varepsilon,t}''=\frac{1}{1-2\varepsilon}\mathbb{I}_{(t+\varepsilon,t+1-\varepsilon)}\ast\varrho_\varepsilon.\]
Hence for any fixed $\varepsilon\in(0,\frac{1}{4})$ and any fixed $t\in(-\infty,-1)$, $\sigma_{\varepsilon,t}$ is a smooth increasing convex function, $\sigma_{\varepsilon,t}(0)=0$, $0\leq\sigma_{\varepsilon,t}'\leq1$, $0\leq\sigma_{\varepsilon,t}''\leq\frac{1}{1-2\varepsilon}$, $\sigma_{\varepsilon,t}'=0$ on $(-\infty,t+\frac{3}{4}\varepsilon]$, and $\sigma_{\varepsilon,t}'=1$ on $[t+1-\frac{3}{4}\varepsilon,+\infty)$. In particular,
\begin{displaymath}
\sigma_{\varepsilon,t}(s)=\begin{cases}
a_{\varepsilon,t}\geq t &\quad \mathrm{if}\quad s\leq t\\
s &\quad \mathrm{if}\quad s\geq t+1,
\end{cases}
\end{displaymath}
where $a_{\varepsilon,t}$ is a constant depending only on $\varepsilon$ and $t$.

Since $\sup\limits_{\Omega}\psi<\alpha_0$ for any $\Omega\subset\subset X$ by assumption, we have
\[\alpha_k:=\sup_{\overline{X_{k+1}}}\psi<\alpha_0.\]
Then it follows from the property $(iii)$ in Subsection \ref{subsection 2} that
\[\sup_{\overline{X_k}}\psi_\rho\leq\alpha_k,\quad\forall\rho\in(0,\rho_1).\]

Let $t_0:=\min\{\sup\limits_{\overline{X_3}}\psi-1,-1\}$. Then for any $\varepsilon\in(0,\frac{1}{4})$, any $t\in(-\infty,t_0)$ and any $\rho\in(0,\rho_1)$, we have
\begin{equation}\label{ie:sigma-psi}
t\leq\sigma_{\varepsilon,t}(\psi_\rho)\leq\sigma_{\varepsilon,t}(\alpha_k)=\alpha_k\quad\text{on}\quad\overline{X_k}.
\end{equation}

Let $\zeta$ and $\chi$ be the solution to the following system of ODEs
defined on $(-\infty,\alpha_0)$:
\begin{numcases}{}
\chi(t)\zeta'(t)-\chi'(t)=1,\label{jet ode1}\\
\bigg(\chi(t)+\frac{\big(\chi'(t)\big)^2}{\chi(t)
\zeta''(t)-\chi''(t)}\bigg)e^{\zeta(t)}=\bigg(\frac{1}{\alpha R(\alpha_0)}+C_R\bigg)R(t),\label{jet ode2}
\end{numcases}
where we assume that $\zeta$ and $\chi$ are both
smooth on $(-\infty,\alpha_0)$, and that $\inf\limits_{t<\alpha_0}\zeta(t)=0$, $\inf\limits_{t<\alpha_0}\chi(t)\geq\frac{1}{\alpha}$,
$\zeta'>0$ and $\chi'<0$ on
$(-\infty,\alpha_0)$. By the similar calculation as in \cite{Guan-Zhou2013b} or \cite{Zhou-Zhu2015}, we can solve the system of ODEs and get the solution
\begin{numcases}{}
\zeta(t)=\log\bigg(\frac{1}{\alpha R(\alpha_0)}+C_R\bigg)-\log\bigg(\frac{1}{\alpha R(\alpha_0)}+\int_t^{\alpha_0}\frac{dt_1}{R(t_1)}\bigg),\nonumber\\
\chi(t)=\frac{\int_t^{\alpha_0}\big(\frac{1}{\alpha R(\alpha_0)}+\int_{t_2}^{\alpha_0}\frac{dt_1}{R(t_1)}\big)dt_2+\frac{1}{\alpha^2R(\alpha_0)}}
{\frac{1}{\alpha R(\alpha_0)}+\int_t^{\alpha_0}\frac{dt_1}{R(t_1)}}.\nonumber
\end{numcases}

Let $h_{\varepsilon,t,\rho}$ be the new metric on the line bundle $L$ over
$X_k\setminus\Sigma_\rho$ defined by
\[h_{\varepsilon,t,\rho}:=
h_0e^{-\phi_\rho-\psi_\rho-\zeta(\sigma_{\varepsilon,t}(\psi_\rho))}.\]

Let $\tau_{\varepsilon,t,\rho}:=\chi(\sigma_{\varepsilon,t}(\psi_\rho))$ and
$A_{\varepsilon,t,\rho}:=\frac{(\chi'(\sigma_{\varepsilon,t}(\psi_\rho)))^2}{\chi(\sigma_{\varepsilon,t}(\psi_\rho))
\zeta''(\sigma_{\varepsilon,t}(\psi_\rho))-\chi''(\sigma_{\varepsilon,t}(\psi_\rho))}$. Set
$\mathrm{B}_{\varepsilon,t,\rho} =[\Theta_{\varepsilon,t,\rho},\,\Lambda
]$ on $X_k\setminus\Sigma_\rho$, where
\[\Theta_{\varepsilon,t,\rho}:=\tau_{\varepsilon,t,\rho}\sqrt{-1}
\Theta_{L,h_{\varepsilon,t,\rho}}
-\sqrt{-1}\partial\bar\partial\tau_{\varepsilon,t,\rho}
-\sqrt{-1}\frac{\partial\tau_{\varepsilon,t,\rho}\wedge
\bar\partial\tau_{\varepsilon,t,\rho}}{A_{\varepsilon,t,\rho}}.\]
We want to prove
\begin{equation}\label{ie:jet twisted cuvature}
\Theta_{\varepsilon,t,\rho}\big|_{X_k\setminus\Sigma_\rho}\geq
\sigma_{\varepsilon,t}''(\psi_\rho)\sqrt{-1}\partial\psi_\rho\wedge
\bar\partial\psi_\rho-2\chi(\sigma_{\varepsilon,t}(\psi_\rho))\delta_\rho\omega.
\end{equation}

It follows from \eqref{jet ode1} that
\begin{eqnarray*}
& &\Theta_{\varepsilon,t,\rho}\big|_{X_k\setminus\Sigma_\rho}\\
&=&\chi(\sigma_{\varepsilon,t}(\psi_\rho))\big(\sqrt{-1}\Theta_{L,h_0}
+\sqrt{-1}\partial\bar\partial
\phi_\rho+\sqrt{-1}\partial\bar\partial\psi_\rho\big)\\
& &+\big(\chi(\sigma_{\varepsilon,t}(\psi_\rho))
\zeta'(\sigma_{\varepsilon,t}(\psi_\rho))
-\chi'(\sigma_{\varepsilon,t}(\psi_\rho))\big)
\sqrt{-1}\partial\bar\partial\sigma_{\varepsilon,t}(\psi_\rho)\\
&=&\chi(\sigma_{\varepsilon,t}(\psi_\rho))\big(\sqrt{-1}\Theta_{L,h_0}
+\sqrt{-1}\partial\bar\partial
\phi_\rho+\sqrt{-1}\partial\bar\partial\psi_\rho\big)+\sqrt{-1}\partial\bar\partial
(\sigma_{\varepsilon,t}(\psi_\rho))\\
&=&\chi(\sigma_{\varepsilon,t}(\psi_\rho))\big(\sqrt{-1}\Theta_{L,h_0}
+\sqrt{-1}\partial\bar\partial
\phi_\rho+\sqrt{-1}\partial\bar\partial\psi_\rho\big)\\
& &+\sigma_{\varepsilon,t}'(\psi_\rho)\sqrt{-1}\partial\bar\partial\psi_\rho
+\sigma_{\varepsilon,t}''(\psi_\rho)\sqrt{-1}\partial\psi_\rho\wedge
\bar\partial\psi_\rho.
\end{eqnarray*}
Since $0\leq\sigma_{\varepsilon,t}'\leq1$ on $\mathbb{R}$ and $\chi\geq\inf\limits_{t<\alpha_0}\chi(t)\geq\frac{1}{\alpha}$ on $(-\infty,\alpha_0)$, it follows from the properties $(iv)$ and $(v)$ in Subsection \ref{subsection 2} that
\begin{eqnarray*}
& &\chi(\sigma_{\varepsilon,t}(\psi_\rho))\big(\sqrt{-1}\Theta_{L,h_0}
+\sqrt{-1}\partial\bar\partial
\phi_\rho+\sqrt{-1}\partial\bar\partial\psi_\rho\big)
+\sigma_{\varepsilon,t}'(\psi_\rho)\sqrt{-1}\partial\bar\partial\psi_\rho\\
&=&\chi(\sigma_{\varepsilon,t}(\psi_\rho))\big(\sqrt{-1}\Theta_{L,h_0}
+\sqrt{-1}\partial\bar\partial
\phi_\rho+\sqrt{-1}\partial\bar\partial\psi_\rho+2\delta_\rho\omega\big)
-2\chi(\sigma_{\varepsilon,t}(\psi_\rho))\delta_\rho\omega\\
& &+\frac{\sigma_{\varepsilon,t}'(\psi_\rho)}{\alpha}\cdot\alpha\sqrt{-1}\partial\bar\partial\psi_\rho\\
&\geq&\frac{\sigma_{\varepsilon,t}'(\psi_\rho)}{\alpha}
\bigg(\sqrt{-1}\Theta_{L,h_0}
+\sqrt{-1}\partial\bar\partial
\phi_\rho+\sqrt{-1}\partial\bar\partial\psi_\rho+2\delta_\rho\omega+\alpha\sqrt{-1}\partial\bar\partial\psi_\rho\bigg)\\
& &-2\chi(\sigma_{\varepsilon,t}(\psi_\rho))\delta_\rho\omega\\
&\geq&-2\chi(\sigma_{\varepsilon,t}(\psi_\rho))\delta_\rho\omega
\end{eqnarray*}
on $X_k\setminus\Sigma_\rho$. Hence we get \eqref{ie:jet twisted cuvature} as desired.

Let $\delta(t)$ be a positive increasing function defined on $(-\infty,t_0)$, such that
\[\lim\limits_{t\rightarrow-\infty}\delta(t)=0.\]
The explicit expression of $\delta(t)$ will be determined later (cf. \eqref{e:def delta t}).

We choose an increasing family of
positive numbers
$\{\rho_t\}_{t\in(-\infty,t_0)}$ such that $\rho_t<\rho_1$ for any $t$,
$\lim\limits_{t\rightarrow-\infty}\rho_t=0$, and
\begin{equation}\label{ie:jet delta1}
2\chi(t)\delta_{\rho_t}<\delta(t)\text{ for any }t.
\end{equation}

Since $\chi$ is decreasing, we have
$\chi(\sigma_{\varepsilon,t}(\psi_\rho))\leq \chi(t)$
on $X_k$ by \eqref{ie:sigma-psi}. Then it follows from \eqref{ie:jet twisted cuvature} and
\eqref{ie:jet delta1} that
\[\Theta_{\varepsilon,t,\rho}
\big|_{X_k\setminus\Sigma_\rho}\geq
\sigma_{\varepsilon,t}''(\psi_\rho)\sqrt{-1}\partial\psi_\rho\wedge
\bar\partial\psi_\rho-\delta(t)\omega,\quad\forall\,\rho\in(0,\rho_t].\]
Hence
\begin{equation}\label{ie:jet B-estimate}
\mathrm{B}_{\varepsilon,t,\rho}
+\delta(t)\mathrm{I}
\geq[\sigma_{\varepsilon,t}''(\psi_\rho)\sqrt{-1}\partial\psi_\rho\wedge
\bar\partial\psi_\rho,\,\Lambda]=\sigma_{\varepsilon,t}''(\psi_\rho)\mathrm{T}_{\bar\partial\psi_\rho}
\mathrm{T}_{\bar\partial\psi_\rho}^*\geq0
\end{equation}
holds on $X_k\setminus\Sigma_\rho$ for any $\rho\in(0,\rho_t]$ as an operator on
$(n,q+1)$-forms, where $\mathrm{T}_{\bar\partial\psi_\rho}$ denotes the
operator $\bar\partial\psi_\rho\wedge\bullet$ and
$\mathrm{T}_{\bar\partial\psi_\rho}^*$ is its Hilbert adjoint
operator.

\subsection{Construction of suitably truncated forms and solving
$\bar\partial$ globally with $L^2$ estimates.}\label{jet-subsection 4} \quad

Define $g_{\varepsilon,t,\rho}=\bar\partial
\big((1-\sigma_{\varepsilon,t}'(\psi_\rho))
\tilde{f}\big)$ on $X_k\setminus\Sigma_\rho$, where $\tilde{f}$ is as in Subsection \ref{jet-subsection 1}. Then $\bar\partial g_{\varepsilon,t,\rho}=0$ on $X_k\setminus\Sigma_\rho$ and
\[g_{\varepsilon,t,\rho}=g_{1,\varepsilon,t,\rho}+g_{2,\varepsilon,t,\rho},\]
where $g_{1,\varepsilon,t,\rho}:=-\sigma_{\varepsilon,t}''(\psi_\rho)\bar\partial\psi_\rho\wedge\tilde{f}$ and
$g_{2,\varepsilon,t,\rho}:=(1-\sigma_{\varepsilon,t}'(\psi_\rho))\bar\partial \tilde{f}$.

The Cauchy-Schwarz inequality and \eqref{ie:jet B-estimate} imply that, on $X_k\setminus\Sigma_\rho$,
\begin{eqnarray}
& &\langle(\mathrm{B}_{\varepsilon,t,\rho}+2\delta(t) \mathrm{I})^{-1}
g_{\varepsilon,t,\rho},g_{\varepsilon,t,\rho}
\rangle_{\omega,h_{\varepsilon,t,\rho}}\label{ie:jet estimate with delta}\\
&\leq&(1+\varepsilon)\langle(\mathrm{B}_{\varepsilon,t,\rho}+2\delta(t)
\mathrm{I})^{-1} g_{1,\varepsilon,t,\rho},g_{1,\varepsilon,t,\rho}
\rangle_{\omega,h_{\varepsilon,t,\rho}}\nonumber\\
& &+(1+\frac{1}{\varepsilon})\langle(\mathrm{B}_{\varepsilon,t,\rho}+2\delta(t)
\mathrm{I})^{-1} g_{2,\varepsilon,t,\rho},g_{2,\varepsilon,t,\rho}
\rangle_{\omega,h_{\varepsilon,t,\rho}}\nonumber\\
&\leq&(1+\varepsilon)\langle(\mathrm{B}_{\varepsilon,t,\rho}+\delta(t)
\mathrm{I})^{-1} g_{1,\varepsilon,t,\rho},g_{1,\varepsilon,t,\rho}
\rangle_{\omega,h_{\varepsilon,t,\rho}}\nonumber\\
& &+(1+\frac{1}{\varepsilon})\langle\frac{1}{\delta(t)}
g_{2,\varepsilon,t,\rho},g_{2,\varepsilon,t,\rho}
\rangle_{\omega,h_{\varepsilon,t,\rho}}\nonumber
\end{eqnarray}
for any $\rho\in(0,\rho_t]$.

Since $\zeta\geq0$ and $\phi_\rho\geq\phi$ on $X_k$, we obtain from \eqref{ie:jet B-estimate} that
\begin{eqnarray*}
& &\int_{X_k\setminus\Sigma_\rho}\langle
(\mathrm{B}_{\varepsilon,t,\rho}+\delta(t) \mathrm{I})^{-1}
g_{1,\varepsilon,t,\rho},g_{1,\varepsilon,t,\rho}
\rangle_{\omega,h_{\varepsilon,t,\rho}}dV_{X,\omega}\\
&\leq&\int_{X_k}
\sigma_{\varepsilon,t}''(\psi_\rho)|\tilde{f}|^2_{\omega,h_0}
e^{-\phi_\rho-\psi_\rho}dV_{X,\omega}\\
&\leq&\frac{1}{1-2\varepsilon}\int_{X_k}
|\tilde{f}|^2_{\omega,h_0}
e^{-\phi-\psi_\rho}\mathbb{I}_{\{t+\frac{3}{4}\varepsilon<\psi_\rho<t+1-\frac{3}{4}\varepsilon\}}dV_{X,\omega}
\end{eqnarray*}
for any $\rho\in(0,\rho_t]$, where $\mathbb{I}_{\{t+\frac{3}{4}\varepsilon<\psi_\rho<t+1-\frac{3}{4}\varepsilon\}}$ is the characteristic function associated to the set $\{x\in X_k:\,t+\frac{3}{4}\varepsilon<\psi_\rho(x)<t+1-\frac{3}{4}\varepsilon\}$.

Since $\zeta\geq0$ and
$\phi_\rho+\psi_\rho\geq\phi+\psi$ on
$X_k\setminus\Sigma$, we get
\begin{eqnarray*}
& &\int_{X_k\setminus\Sigma_\rho}\langle\frac{1}
{\delta(t)}g_{2,\varepsilon,t,\rho},g_{2,\varepsilon,t,\rho}
\rangle_{\omega,h_{\varepsilon,t,\rho}}dV_{X,\omega}\\
&\leq&\frac{1}{\delta(t)}\int_{X_k}
|\bar\partial \tilde{f}|^2_{\omega,h_0}e^{-\phi-\psi}\mathbb{I}_{\{\psi_\rho<t+1-\frac{3}{4}\varepsilon\}}dV_{X,\omega}.
\end{eqnarray*}

Therefore, it follows from \eqref{ie:jet estimate with delta} that
\begin{eqnarray*}
C_\rho(t)&:=&\int_{X_k\setminus\Sigma_\rho}\langle(\mathrm{B}_{\varepsilon,t,\rho}
+2\delta(t) \mathrm{I})^{-1}
g_{\varepsilon,t,\rho},g_{\varepsilon,t,\rho}
\rangle_{\omega,h_{\varepsilon,t,\rho}}dV_{X,\omega}\\
&\leq&\frac{1+\varepsilon}{1-2\varepsilon}\int_{X_k}
|\tilde{f}|^2_{\omega,h_0}
e^{-\phi-\psi_\rho}\mathbb{I}_{\{t+\frac{3}{4}\varepsilon<\psi_\rho<t+1-\frac{3}{4}\varepsilon\}}dV_{X,\omega}\\
& &+(1+\frac{1}{\varepsilon})\frac{1}{\delta(t)}\int_{X_k}
|\bar\partial \tilde{f}|^2_{\omega,h_0}e^{-\phi-\psi}\mathbb{I}_{\{\psi_\rho<t+1-\frac{3}{4}\varepsilon\}}dV_{X,\omega}
\end{eqnarray*}
for any $\rho\in(0,\rho_t]$.

Since $\lim\limits_{\rho\rightarrow0}\psi_\rho=\psi$ on $X_k\setminus\Sigma$, it follows from \eqref{ie:jet f int finite}, \eqref{ie:jet dbar f int finite} and Fatou's lemma that
\begin{eqnarray}
C(t)&:=&\varlimsup\limits_{\rho\rightarrow0}C_\rho(t)\label{ie:jet estimate-C-t}\\
&\leq&\frac{1+\varepsilon}{1-2\varepsilon}\int_{X_k}
|\tilde{f}|^2_{\omega,h_0}
e^{-\phi-\psi}\mathbb{I}_{\{t+\frac{3}{4}\varepsilon\leq\psi\leq t+1-\frac{3}{4}\varepsilon\}}dV_{X,\omega}\nonumber\\
& &+(1+\frac{1}{\varepsilon})\frac{1}{\delta(t)}\int_{X_k}
|\bar\partial\tilde{f}|^2_{\omega,h_0}e^{-\phi-\psi}\mathbb{I}_{\{\psi\leq t+1-\frac{3}{4}\varepsilon\}}dV_{X,\omega}.\nonumber
\end{eqnarray}

Let
\begin{equation}\label{e:def delta t}
\delta(t):=\bigg(\int_{X_k}|\bar\partial\tilde{f}|^2_{\omega,h_0}e^{-\phi-\psi}\mathbb{I}_{\{\psi<t+1\}}dV_{X,\omega}\bigg)^{\frac{1}{2}}.
\end{equation}
Then
\[C(t)\leq\frac{1+\varepsilon}{1-2\varepsilon}\int_{X_k}
|\tilde{f}|^2_{\omega,h_0}
e^{-\phi-\psi}\mathbb{I}_{\{t<\psi<t+1\}}dV_{X,\omega}+(1+\frac{1}{\varepsilon})\delta(t).\]
Hence
\begin{equation}\label{e:jet limit-C-t}
\varlimsup\limits_{t\rightarrow-\infty}C(t)\leq\frac{1+\varepsilon}{1-2\varepsilon}\int_Y|f|^2_{\omega,h} dV_{X,\omega}[\psi].
\end{equation}

Since $X_k$ is a weakly pseudoconvex K\"{a}hler manifold, $X_k$ carries a complete K\"{a}hler metric. Hence $X_k\setminus\Sigma_\rho$ carries a complete K\"{a}hler metric by Lemma \ref{l:Demailly1982-complete-metric}. Then by Lemma \ref{l:Demailly-non complete metric}, there exists an $L$-valued $(n,0)$-form $u_{k,\varepsilon,t,\rho}$ and an $L$-valued $(n,1)$-form $v_{k,\varepsilon,t,\rho}$ such that
\begin{equation}\label{e:jet dbar}
\bar\partial u_{k,\varepsilon,t,\rho}
+\sqrt{2\delta(t)}v_{k,\varepsilon,t,\rho} =g_{\varepsilon,t,\rho}\quad\text{on}\quad X_k\setminus\Sigma_\rho
\end{equation}
and
\begin{eqnarray*}
& &\int_{X_k\setminus\Sigma_\rho}\frac{|u_{k,\varepsilon,t,\rho}|^2
_{\omega,h_0}e^{-\phi_\rho-\psi_\rho-\zeta(\sigma_{\varepsilon,t}(\psi_\rho))}}
{\tau_{\varepsilon,t,\rho}+A_{\varepsilon,t,\rho}}dV_{X,\omega}\\
& &+\int_{X_k\setminus\Sigma_\rho}|v_{k,\varepsilon,t,\rho}|^2_{\omega,h_0}
e^{-\phi_\rho-
\psi_\rho-\zeta(\sigma_{\varepsilon,t}(\psi_\rho))}dV_{X,\omega}\\
&\leq& C_\rho(t)
\end{eqnarray*}
for any $\rho\in(0,\rho_t]$.

Since \eqref{jet ode2} implies that
\[\frac{e^{-\zeta(\sigma_{\varepsilon,t}(\psi_\rho))}}{\tau_{\varepsilon,t,\rho}+A_{\varepsilon,t,\rho}}=\frac{1}{\big(\frac{1}{\alpha R(\alpha_0)}+C_R\big)R(\sigma_{\varepsilon,t}(\psi_\rho))}\]
and \eqref{ie:sigma-psi} implies that
\[e^{-\zeta(\sigma_{\varepsilon,t}(\psi_\rho))}\geq e^{-\zeta(\alpha_k)},\]
we get
\begin{equation}\label{ie:jet estimate-u-and-h}
\int_{X_k}\frac{|u_{k,\varepsilon,t,\rho}|^2
_{\omega,h_0}e^{-\phi_\rho-\psi_\rho}dV_{X,\omega}}{\big(\frac{1}{\alpha R(\alpha_0)}+C_R\big)R(\sigma_{\varepsilon,t}(\psi_\rho))}+e^{-\zeta(\alpha_k)}\int_{X_k}|v_{k,\varepsilon,t,\rho}|^2_{\omega,h_0}
e^{-\phi_\rho-
\psi_\rho}dV_{X,\omega}
\leq C_\rho(t)
\end{equation}
for any $\rho\in(0,\rho_t]$.

Since $R$ is decreasing near $-\infty$, \eqref{ie:sigma-psi} implies that $R(\sigma_{\varepsilon,t}(\psi_\rho))\leq R(t)$ for any $\varepsilon\in(0,\frac{1}{4})$, any $t\in(-\infty,t_0)$ and any $\rho\in(0,\rho_t)$. Then \eqref{ie:jet estimate-u-and-h} and the property $\sup\limits_{\overline{X_k}\setminus\Sigma}(\phi_\rho+\psi_\rho)<+\infty$ imply that
$u_{k,\varepsilon,t,\rho}\in
L^2$ and $v_{k,\varepsilon,t,\rho}\in
L^2$.

The property $\bar\partial\psi_\rho\in L^1$ on $X_k$ implies that $g_{\varepsilon,t,\rho}=g_{1,\varepsilon,t,\rho}+g_{2,\varepsilon,t,\rho}\in L^1$ on $X_k$. Then it follows from
\eqref{e:jet dbar} and Lemma \ref{l:extension} that
\begin{equation}\label{e:jet dbar on X_k}
\bar\partial u_{k,\varepsilon,t,\rho}
+\sqrt{2\delta(t)}v_{k,\varepsilon,t,\rho} =g_{\varepsilon,t,\rho}=\bar\partial
\big((1-\sigma_{\varepsilon,t}'(\psi_\rho))
\tilde{f}\big)\quad\text{on}\quad X_k.
\end{equation}

Since $\lim\limits_{\rho\rightarrow0}\psi_\rho=\psi$ on $X_k\setminus\Sigma$, $(1-\sigma_{\varepsilon,t}'(\psi_\rho))
\tilde{f}$ converges to $(1-\sigma_{\varepsilon,t}'(\psi))
\tilde{f}$ in $L^2$ as $\rho\rightarrow0$ by Lebesgue's dominated convergence theorem. Hence $\bar\partial
\big((1-\sigma_{\varepsilon,t}'(\psi_\rho))
\tilde{f}\big)$ converges to $\bar\partial
\big((1-\sigma_{\varepsilon,t}'(\psi))
\tilde{f}\big)$ in the sense of currents as $\rho\rightarrow0$.

Let $t$ be fixed and let $\{\rho_j\}_{j=2}^{+\infty}$ be a decreasing sequence of positive numbers such that $\rho_2<\rho_t$ and $\lim\limits_{j\rightarrow+\infty}\rho_j=0$.

Since $\lim\limits_{j\rightarrow+\infty}\psi_{\rho_j}=\psi$ on $\overline{X_k}\setminus\Sigma$ and $\phi_{\rho_j}+\psi_{\rho_j}$ decreases to $\phi+\psi$ on $\overline{X_k}\setminus\Sigma$ as $j\rightarrow+\infty$,
\[\frac{e^{-\phi_{\rho_j}-\psi_{\rho_j}}}{\sup\limits_{p\geq j}R(\sigma_{\varepsilon,t}(\psi_{\rho_p}))}\]
increases to $\frac{e^{-\phi-\psi}}{R(\sigma_{\varepsilon,t}(\psi))}$ on $\overline{X_k}\setminus\Sigma$ as $j\rightarrow+\infty$.

By extracting weak limits of $\{u_{k,\varepsilon,t,\rho_j}\}_{j=2}^{+\infty}$ as $j\rightarrow+\infty$, it follows from \eqref{ie:jet estimate-u-and-h} and the diagonal argument that there exists an $L^2$ $L$-valued $(n,0)$-form $u_{k,\varepsilon,t}$ such that a subsequence $\{u_{k,\varepsilon,t,\rho_{j_r}}\}_{r=1}^{+\infty}$ of $\{u_{k,\varepsilon,t,\rho_j}\}_{j=2}^{+\infty}$ converges to $u_{k,\varepsilon,t}$ weakly in \[L^2_{(n,0)}\bigg(X_k,\,\frac{e^{-\phi_{\rho_i}-\psi_{\rho_i}}dV_{X,\omega}}{\sup\limits_{p\geq i}R(\sigma_{\varepsilon,t}(\psi_{\rho_p}))}\bigg)\]
for any positive integer $i$ as $r\rightarrow+\infty$.

Hence the Banach-Steinhaus Theorem implies that, for any positive integer $i$,
\begin{eqnarray*}
\int_{X_k}\frac{|u_{k,\varepsilon,t}|^2_{\omega,h_0}e^{-\phi_{\rho_i}-\psi_{\rho_i}}dV_{X,\omega}}{\sup\limits_{p\geq i}R(\sigma_{\varepsilon,t}(\psi_{\rho_p}))}&\leq&
\varliminf_{r\rightarrow+\infty}\int_{X_k}\frac{|u_{k,\varepsilon,t,\rho_{j_r}}|^2_{\omega,h_0}e^{-\phi_{\rho_i}-\psi_{\rho_i}}
dV_{X,\omega}}{\sup\limits_{p\geq i}R(\sigma_{\varepsilon,t}(\psi_{\rho_p}))}\\
&\leq&\varliminf_{r\rightarrow+\infty}\int_{X_k}\frac{|u_{k,\varepsilon,t,\rho_{j_r}}|^2_{\omega,h_0}e^{-\phi_{\rho_{j_r}}-\psi_{\rho_{j_r}}}
dV_{X,\omega}}{\sup\limits_{p\geq j_r}R(\sigma_{\varepsilon,t}(\psi_{\rho_p}))}.
\end{eqnarray*}
Then Fatou's lemma implies that
\begin{eqnarray*}
\int_{X_k}\frac{|u_{k,\varepsilon,t}|^2_{\omega,h_0}e^{-\phi-\psi}dV_{X,\omega}}{R(\sigma_{\varepsilon,t}(\psi))}&\leq&
\varliminf_{i\rightarrow+\infty}\int_{X_k}\frac{|u_{k,\varepsilon,t}|^2_{\omega,h_0}e^{-\phi_{\rho_i}-\psi_{\rho_i}}dV_{X,\omega}}{\sup\limits_{p\geq i}R(\sigma_{\varepsilon,t}(\psi_{\rho_p}))}\\
&\leq&\varliminf_{r\rightarrow+\infty}\int_{X_k}\frac{|u_{k,\varepsilon,t,\rho_{j_r}}|^2_{\omega,h_0}e^{-\phi_{\rho_{j_r}}-\psi_{\rho_{j_r}}}
dV_{X,\omega}}{\sup\limits_{p\geq j_r}R(\sigma_{\varepsilon,t}(\psi_{\rho_p}))}.
\end{eqnarray*}

Similar weak limit argument can also be applied to subsequences of $\{v_{k,\varepsilon,t,\rho_j}\}_{j=2}^{+\infty}$.

In conclusion, by \eqref{e:jet dbar on X_k}, \eqref{ie:jet estimate-u-and-h} and \eqref{ie:jet estimate-C-t}, there exist $L^2$ forms $u_{k,\varepsilon,t}$ and $v_{k,\varepsilon,t}$ such that
\begin{equation}\label{e:jet dbar on X_k limit-1}
\bar\partial u_{k,\varepsilon,t}
+\sqrt{2\delta(t)}v_{k,\varepsilon,t}=\bar\partial
\big((1-\sigma_{\varepsilon,t}'(\psi))
\tilde{f}\big)\quad\text{on}\quad X_k
\end{equation}
and
\begin{equation}\label{ie:jet estimate-u-and-h no rho}
\int_{X_k}\frac{|u_{k,\varepsilon,t}|^2
_{\omega,h_0}e^{-\phi-\psi}dV_{X,\omega}}{\big(\frac{1}{\alpha R(\alpha_0)}+C_R\big)R(\sigma_{\varepsilon,t}(\psi))}+e^{-\zeta(\alpha_k)}\int_{X_k}|v_{k,\varepsilon,t}|^2_{\omega,h_0}
e^{-\phi-
\psi}dV_{X,\omega}
\leq C(t).
\end{equation}

\eqref{e:jet limit-C-t} and \eqref{ie:jet estimate-u-and-h no rho} imply that
\begin{equation}\label{e:jet error term limit-1}
\varlimsup\limits_{t\rightarrow-\infty}\int_{X_k}|\sqrt{2\delta(t)}v_{k,\varepsilon,t}|^2_{\omega,h_0}e^{-\phi-\psi}dV_{X,\omega}
\leq\varlimsup\limits_{t\rightarrow-\infty}2e^{\zeta(\alpha_k)}\delta(t)C(t)=0.
\end{equation}

Define $F_{k,\varepsilon,t}=-u_{k,\varepsilon,t}+(1-\sigma_{\varepsilon,t}'(\psi))
\tilde{f}$ on $X_k$. Then \eqref{e:jet dbar on X_k limit-1} implies that
\begin{equation}\label{e:jet dbar F}
\bar\partial F_{k,\varepsilon,t}=\sqrt{2\delta(t)}v_{k,\varepsilon,t}\quad\text{on}\quad X_k.
\end{equation}

Since $R$ is decreasing near $-\infty$, $R(\sigma_{\varepsilon,t}(s))\leq R(s)$ for all $s\in(-\infty,\alpha_0)$ when $t$ is small enough. Then \eqref{ie:jet estimate-u-and-h no rho} implies that
\begin{eqnarray}
& &\int_{X_k}\frac{|F_{k,\varepsilon,t}|^2_{\omega,h_0}e^{-\phi}}{e^\psi R(\psi)}dV_{X,\omega}\label{ie:jet F estimate-1}\\
&\leq&(1+\varepsilon)\int_{X_k}\frac{|u_{k,\varepsilon,t}|^2_{\omega,h_0}e^{-\phi}}{e^\psi R(\sigma_{\varepsilon,t}(\psi))}dV_{X,\omega}
+\frac{1+\varepsilon}{\varepsilon}\int_{X_k}\frac{|(1-\sigma_{\varepsilon,t}'(\psi))\tilde{f}|^2_{\omega,h_0}e^{-\phi}}{e^\psi R(\psi)}dV_{X,\omega}\nonumber\\
&\leq&(1+\varepsilon)\bigg(\frac{1}{\alpha R(\alpha_0)}+C_R\bigg)C(t)+\widetilde{C}(t)\nonumber
\end{eqnarray}
when $t$ is small enough, where
\[\widetilde{C}(t):=\frac{1+\varepsilon}{\varepsilon}\int_{X_k}\frac{|\tilde{f}|^2_{\omega,h_0}e^{-\phi}\mathbb{I}_{\{\psi<t+1\}}}{e^\psi R(\psi)}dV_{X,\omega}.\]

Since \eqref{ie:jet-extension-thm-f-finite} implies that
\[\varlimsup_{t\rightarrow-\infty}\int_{X_k}|\tilde{f}|^2_{\omega,h_0}e^{-\phi-\psi}\mathbb{I}_{\{t<\psi<t+1\}}dV_{X,\omega}<+\infty,\]
there exists a positive number $C_1$ such that
\[\int_{X_k}|\tilde{f}|^2_{\omega,h_0}e^{-\phi-\psi}\mathbb{I}_{\{t-j<\psi<t+1-j\}}dV_{X,\omega}\leq C_1\]
for all nonnegative integer $j$ when $t$ is small enough.

Since $R$ is decreasing near $-\infty$, we get
\begin{eqnarray*}
\widetilde{C}(t)&\leq&\frac{1+\varepsilon}{\varepsilon}\sum_{j=0}^{+\infty}\frac{1}{R(t+1-j)}
\int_{X_k}|\tilde{f}|^2_{\omega,h_0}e^{-\phi-\psi}\mathbb{I}_{\{t-j<\psi<t+1-j\}}dV_{X,\omega}\\
&\leq&\frac{1+\varepsilon}{\varepsilon}C_1\int_{-\infty}^{t+2}\frac{1}{R(s)}ds
\end{eqnarray*}
when $t$ is small enough. Hence
\begin{equation}\label{e:jet-lim-sencond C-t}
\lim_{t\rightarrow-\infty}\widetilde{C}(t)=0.
\end{equation}

Since $\varlimsup\limits_{s\rightarrow-\infty} e^sR(s)<+\infty$, we obtain from \eqref{ie:jet F estimate-1}, \eqref{e:jet limit-C-t} and \eqref{e:jet-lim-sencond C-t} that
\begin{equation}\label{ie:jet F estimate-2}
\int_{X_k}|F_{k,\varepsilon,t}|^2_{\omega,h_0}dV_{X,\omega}\leq C_2
\end{equation}
for some positive number $C_2$ independent of $t$ when $t$ is small enough.

By extracting weak limits of $\{F_{k,\varepsilon,t}\}_{t\in(-\infty,t_0)}$ as $t\rightarrow-\infty$, it follows from \eqref{ie:jet F estimate-1}, \eqref{ie:jet F estimate-2}, \eqref{e:jet limit-C-t} and \eqref{e:jet-lim-sencond C-t} that there exists a sequence of negative numbers $\{t_j\}_{j=1}^{+\infty}$ and an $L$-valued $(n,0)$-form $F_{k,\varepsilon}$ such that $\lim\limits_{j\rightarrow+\infty}t_j=-\infty$, $F_{k,\varepsilon,t_j}\rightharpoonup F_{k,\varepsilon}$ weakly in both $L^2_{(n,0)}(X_k,\frac{e^{-\phi}dV_{X,\omega}}{e^\psi R(\psi)})$ and $L^2_{(n,0)}(X_k,dV_{X,\omega})$, and
\begin{equation}\label{ie:jet F estimate-3}
\int_{X_k}\frac{|F_{k,\varepsilon}|^2_{\omega,h_0}e^{-\phi}}{e^\psi R(\psi)}dV_{X,\omega}\leq\frac{(1+\varepsilon)^2}{1-2\varepsilon}\bigg(\frac{1}{\alpha R(\alpha_0)}+C_R\bigg)\int_Y|f|^2_{\omega,h} dV_{X,\omega}[\psi].
\end{equation}

It follows from \eqref{e:jet error term limit-1} and \eqref{e:jet dbar F} that $\bar\partial F_{k,\varepsilon}=0$ on $X_k$. Thus
\[F_{k,\varepsilon}\in H^0\big(X_k,\mathcal{O}_X(K_X\otimes L)\big).\]

In Subsection \ref{jet-subsection 5}, we will prove that $F_{k,\varepsilon}\in H^0\big(X_k,\mathcal{O}_X(K_X\otimes L)\otimes\mathcal{I}'_{\psi}(h)\big)$ and that $F_{k,\varepsilon}$ maps to $f$ under the morphism $\mathcal{I}'_{\psi}(h)\longrightarrow\mathcal{I}'_{\psi}(h)/\mathcal{I}(he^{-\psi})$.

\subsection{Solving $\bar\partial$ locally with $L^2$ estimates and the end of the proof.}\label{jet-subsection 5} \quad

Let $\{U_i\}_{i\in I}$ be the covering of $X$ as in Subsection \ref{jet-subsection 1}. Let $i\in I$ be any fixed index. Let $V$ be an arbitrary relatively compact coordinate ball contained in $U_i\cap X_k$ such that $L$ is trivial on $\overline{V}$. Then $f_i$, $\tilde{f}$, $u_{k,\varepsilon,t}$, $v_{k,\varepsilon,t}$, $F_{k,\varepsilon,t}$ and $F_{k,\varepsilon}$ can be regarded as forms with values in $\mathbb{C}$ when they are restricted to $\overline{V}$.

It follows from \eqref{ie:jet estimate-u-and-h no rho} that
\[\int_{V}|v_{k,\varepsilon,t}|^2 e^{-\phi-\psi}d\lambda_n \leq
C_3\]
for some positive number $C_3$ independent of $t$ when
$t$ is small enough, where $d\lambda_n$ is the $n$-dimensional Lebesgue measure on $V$.

Since $\bar\partial v_{k,\varepsilon,t}=0$ on $V$ by
\eqref{e:jet dbar on X_k limit-1}, applying Lemma \ref{l:Hormander-dbar} to
the $(n,1)$-form
\[\sqrt{2\delta(t)}v_{k,\varepsilon,t}\in
L^2_{(n,1)}(V,e^{-\phi-\psi}),\] we get an $(n,0)$-form
$w_{k,\varepsilon,t}$ such
that $\bar\partial w_{k,\varepsilon,t}
=\sqrt{2\delta(t)}v_{k,\varepsilon,t}$ on $V$
and
\begin{equation}\label{ie:estimate v1}
\int_{V}|w_{k,\varepsilon,t}|^2 e^{-\phi-\psi}d\lambda_n\leq
C_4\delta(t)
\end{equation}
for some positive number $C_4$ independent of $t$.

Hence
\begin{equation}\label{ie:estimate v2}
\int_{V}|w_{k,\varepsilon,t}|^2 d\lambda_n \leq
C_5\delta(t)
\end{equation}
for some positive number $C_5$ independent of $t$.

Now define $G_{k,\varepsilon,t}=-u_{k,\varepsilon,t}-w_{k,\varepsilon,t}+(1-\sigma_{\varepsilon,t}'(\psi))\tilde{f}$ on $V$. Then
\[G_{k,\varepsilon,t}=F_{k,\varepsilon,t}
-w_{k,\varepsilon,t}\quad\text{and}\quad\bar\partial
G_{k,\varepsilon,t}=0.\]
Hence $G_{k,\varepsilon,t}$ is
holomorphic on $V$. Furthermore, we get from \eqref{ie:jet F estimate-2} and
\eqref{ie:estimate v2} that
\begin{equation}\label{ie:G-estimate}
\int_{V}|G_{k,\varepsilon,t}|^2d\lambda_n\leq C_6
\end{equation}
for some positive number $C_6$ independent of $t$.

Since $\sigma_{\varepsilon,t}\geq t$ and $R$ is decreasing near $-\infty$, we can obtain that $R(\sigma_{\varepsilon,t}(\psi))\leq R(t)$ on $X_k$ when $t$ is small enough. Then we obtain from \eqref{ie:jet estimate-u-and-h no rho} that
\begin{equation*}
\int_{V}|u_{k,\varepsilon,t}|^2 e^{-\phi-\psi}d\lambda_n\leq
C_7R(t)
\end{equation*}
for some positive number $C_7$ independent of $t$.

Therefore, we get
\[\int_{V}|u_{k,\varepsilon,t}
+w_{k,\varepsilon,t}|^2e^{-\phi-\psi}d\lambda_n\leq
2C_7R(t)+2C_4\delta(t).\]

Since
\[G_{k,\varepsilon,t}-f_i=-u_{k,\varepsilon,t}-w_{k,\varepsilon,t}+(1-\sigma_{\varepsilon,t}'(\psi))\sum_{j\in I}\xi_j(f_j-f_i)-\sigma_{\varepsilon,t}'(\psi)f_i\]
and $\sigma_{\varepsilon,t}'(\psi)=0$ on $V\cap\{\psi<t\}$, we have
\begin{equation}\label{e:G multiplier}
G_{k,\varepsilon,t}-f_i\in\mathcal{I}(he^{-\psi})_x,\quad\forall x\in V.
\end{equation}

Since $w_{k,\varepsilon,t_j}\rightarrow 0$ in $L^2$ by \eqref{ie:estimate v2} and
$F_{k,\varepsilon,t_j}\rightharpoonup F_{k,\varepsilon}$ weakly in
$L^2$ as $j\rightarrow+\infty$,
we get $G_{k,\varepsilon,t_j}\rightharpoonup F_{k,\varepsilon}$ weakly in
$L^2$ as $j\rightarrow+\infty$. Hence it follows from
\eqref{ie:G-estimate} and routine arguments with applying Montel's
theorem that a subsequence of
$\{G_{k,\varepsilon,t_j}\}_{j=1}^{+\infty}$ converges to
$F_{k,\varepsilon}$ uniformly on compact subsets of $V$. Then it follows from \eqref{ie:G-estimate}, \eqref{e:G multiplier} and Lemma \ref{l:local generators uniform estimate} that
\begin{equation}\label{e:F multiplier}
F_{k,\varepsilon}-f_i\in\mathcal{I}(he^{-\psi})_x
\end{equation}
for any $x\in V$ and thereby for any $x\in U_i\cap X_k$.

Since $\phi$ is locally bounded above and
$\varlimsup\limits_{s\rightarrow-\infty} e^sR(s)<+\infty$, applying Montel's
theorem and extracting weak limits of $\{F_{k,\varepsilon}\}_{k\geq3,k\in\mathbb{Z},\varepsilon\in(0,\frac{1}{4})}$,
first as $\varepsilon\rightarrow 0$, then as $k\rightarrow+\infty$, we obtain from
\eqref{ie:jet F estimate-3}, \eqref{e:F multiplier} and Lemma \ref{l:local generators uniform estimate} a section $F\in H^0(X,\mathcal{O}_X(K_X\otimes L))$ such that
\[\int_X\frac{|F|^2_{\omega,h}}{e^{\psi}R(\psi)} dV_{X,\omega}
\leq \bigg(\frac{1}{\alpha R(\alpha_0)}+C_R\bigg)\int_Y|f|^2_{\omega,h} dV_{X,\omega}[\psi]\]
and
\[F-f_i\in\mathcal{I}(he^{-\psi})_x,\quad\forall x\in U_i\cap X.\]
Hence $F\in H^0\big(X,\mathcal{O}_X(K_X\otimes L)\otimes\mathcal{I}'_{\psi}(h)\big)$, and $F$ maps to $f$ under the morphism $\mathcal{I}'_{\psi}(h)\longrightarrow\mathcal{I}'_{\psi}(h)/\mathcal{I}(he^{-\psi})$.

The last surjectivity statement in the conclusion of Theorem \ref{t:Zhou-Zhu-optimal-jet-extension} follows by replacing the metric $h$ with a new metric $h_1:=he^{-\Phi(\Psi)}$, where $\Psi$ is the smooth psh exhaustion function on $X$ (cf. Subsection \ref{jet-subsection 1}), and $\Phi:\mathbb{R}\longrightarrow[0,+\infty)$ is some smooth increasing convex function. In fact, in order to obtain a global holomorphic extension $F$, the key point in the whole proof above is the existence of a constant $C_0$ (independent of $\rho$, $t$ and $k$) satisfying $\varlimsup\limits_{t\rightarrow-\infty}C(t)\leq C_0$ (cf. \eqref{e:jet limit-C-t}). It is not hard to see that such a constant $C_0$ exists if $\Phi$ increases fast enough.

In conclusion, Theorem \ref{t:Zhou-Zhu-optimal-jet-extension} is proved.

\bigskip

%%%--------------------------------------------------------

%%% ----------------------------------------------------

%%%----------------------------------------------------

\bibliographystyle{references}
\bibliography{xbib}

\end{document}